\documentclass[a4paper,11pt,pdf]{amsart}
\usepackage{enumerate, amsmath, amsfonts, amssymb, amsthm, thmtools, wasysym, graphics, graphicx, xcolor, frcursive,comment,bbm}

\usepackage{etex}
\usepackage{lscape}
\usepackage{tikz-cd}

\makeatletter
\newcommand{\thickhline}{%
    \noalign {\ifnum 0=`}\fi \hrule height 1pt
    \futurelet \reserved@a \@xhline
}

\definecolor{darkblue}{rgb}{0.0,0,0.7} 

\definecolor{darkred}{rgb}{0.7,0,0} 
\usepackage{hyperref}
\usepackage[all]{xy}
\usepackage[T1]{fontenc}

\usepackage{vmargin}           
\setmarginsrb{2.2cm}{2.5cm}{2.2cm}{2.5cm}{0cm}{0.6cm}{0cm}{0cm}

\usepackage{caption,lipsum}
\usepackage{graphicx}                  
\usepackage{pstricks,pst-plot,pst-text,pst-tree,pst-eps,pst-fill,pst-node,pst-math}
\usepackage{setspace}
\usepackage{multicol}

\newcommand{\darkred}{\color{darkred}} 
\newcommand{\defn}[1]{\emph{\darkred #1}}

\def\S{{\mathcal{S}}}
\def\R{{\mathcal{R}}}

\def\Div{{\sf{Div}}}
\def\S{{\mathcal{S}}}

\usepackage{etex}

\newtheorem{theorem}{Theorem}[section]
\newtheorem{prop}[theorem]{Proposition}
\newtheorem{lemma}[theorem]{Lemma}
\newtheorem{cor}[theorem]{Corollary}

\newtheorem{nota}[theorem]{Notation}

\theoremstyle{definition}
\newtheorem{definition}[theorem]{Definition}
\newtheorem{rmq}[theorem]{Remark}
\newtheorem{exple}[theorem]{Example}
\newtheorem{question}[theorem]{Question}

\numberwithin{equation}{section}

\def\ie{\mbox{\it i.e}.}

\title[A new Garside structure on torus knot groups and some complex braid groups]{A new Garside structure on torus knot groups and some complex braid groups}

\author{Thomas Gobet}
\address{Institut Denis Poisson, CNRS UMR 7350, Faculté des Sciences et Techniques, Université de Tours, Parc de Grandmont, 
37200 TOURS, France}

\begin{document}
\maketitle

\bigskip

\bigskip

\begin{abstract}
Several distinct Garside monoids having torus knot groups as groups of fractions are known. For $n,m\geq 2$ two coprime integers, we introduce a new Garside monoid $\mathcal{M}(n,m)$ having as Garside group the $(n,m)$-torus knot group, thereby generalizing to all torus knot groups a construction that we previously gave for the $(n,n+1)$-torus knot group. As a byproduct, we obtain new Garside structures for the braid groups of a few exceptional complex reflection groups of rank two. Analogous Garside structures are also constructed for a few additional braid groups of exceptional complex reflection groups of rank two which are not isomorphic to torus knot groups, namely for $G_{13}$ and for dihedral Artin groups of even type. 
\end{abstract}

\tableofcontents

\thispagestyle{empty}

\section{Introduction}

A \textit{Garside group} is the group of fractions of a cancellative monoid with a few additional properties--see Section~\ref{facts_garside} below for precise definitions and properties. Such a monoid is called a \textit{Garside monoid}. It is named after a seminal paper of Garside from 1969~\cite{garside_69}, who solved the conjugacy problem in the $n$-strand braid group $\mathcal{B}_n$, and gave a new solution to the word problem and a new method to determine the center of $\mathcal{B}_n$. Roughly speaking, a Garside monoid has good divisibility properties, which can be used to solve the above-mentioned questions in the corresponding Garside group. A Garside group also has other fundamental properties, for instance it is torsion-free. We refer the reader to~\cite{Garside} for basics on Garside monoids and groups, or to Section~\ref{facts_garside} below for a collection of results which we shall use in this paper. Note that a given Garside group can have several non-isomorphic Garside monoids, \ie, non-isomorphic Garside monoids can have isomorphic groups of fractions. 

Dehornoy and Paris~\cite{DP} showed that the approach used by Garside could in fact be generalized to many other known groups, leading to the definition of Garside groups and monoids. Apart from the $n$-strand braid group $\mathcal{B}_n$, basic examples of Garside groups include torus knot groups~\cite[Example 4]{DP}, that is, fundamental groups of the complements of torus knots in $S^3$. These are the only knot groups which can be Garside groups, since a Garside group has a nontrivial center, while it was shown by Burde and Zieschang~\cite{BZ} that torus knot groups are the only knot groups with a nontrivial center.

For all $n,m\geq 2$ such that $n<m$ and $n,m$ are coprime, there is a torus knot $T_{n,m}$ and these knots together with the unknot yield a classification of torus knots (see~\cite{Rolfsen}). Several Garside structures for torus knot groups are known. Indeed, at least three of them already appear in~\cite{DP} (see Presentations~\ref{pres_1}, \ref{pres_2} and \ref{pres_3} below), while Picantin~\cite{Picantin_torus} introduced a Garside structure which is similar to, but distinct in general, from Presentation~\ref{pres_2} below. In a few particular cases, like for $n=3$ and $m=4$, still other Garside structures are known (see~\cite{Picantin_torus, Picantin} and Example~\ref{exple_g12} below).      

Another interesting example is given by the case $n=2, m=3$ which, if we exclude the unknot, is the easiest torus knot: the trefoil knot $T_{2,3}$, whose group is isomorphic to the three-strand braid group $\mathcal{B}_3$ (see for instance~\cite[Section 1.1.4]{KT}). It has several known Garside structures~\cite{BKL, DP, Picantin_torus, Garside}. In a previous work~\cite{Gobet}, motivated by questions related to the submonoid of the $n$-strand braid group generated by $\sigma_1, \sigma_1\sigma_2, \dots, \sigma_1\sigma_2\cdots \sigma_{n-1}$, where the $\sigma_i$'s denote the standard Artin generators of $\mathcal{B}_n$, we generalized the exotic Garside monoid $$\langle \ a, b \ \vert \ aba=b^2 \ \rangle$$ for $\mathcal{B}_{3}$ ($a\mapsto \sigma_1, b\mapsto \sigma_1 \sigma_2)$ to all knot groups of torus knots of the form $T_{n,n+1}$, $n\geq 2$. This yielded a new Garside structure for these groups. The aim of this paper is to generalize this result to all torus knot groups, that is, to construct the aforementioned Garside structure for the knot group $G(n,m)$ of $T_{n,m}$ (see~Presentation~\ref{pres_pratique} and Theorem~\ref{main_text} below). 

While most proofs from the case $T_{n,n+1}$ quite readily generalize to the case $T_{n,m}$, the proof of right-cancellativity is much more involved and technical than in the $T_{n,n+1}$-case, due to the presence of additional relations of a new kind which are not present in the presentations of the $T_{n,n+1}$-case (compare Presentation~\ref{pres_brut} and Example~\ref{exple_old} below). Also, unlike in the $(n,n+1)$-case, an interesting feature of this Garside structure is that the left- and right-lcm of the atoms differ if $2r > n$, where $r$ denotes the rest of the Euclidean division of $m$ by $n$ (see Section~\ref{section_lcm_atoms} below).  

A few torus knot groups are isomorphic to complex braid groups of complex reflection groups of rank two (as noted in~\cite{Bannai}). Some other complex braid groups, like the braid group of the exceptional complex reflection group $G_{13}$, have a presentation which looks very similar to what we may call the \textit{standard} presentation (see Presentation~\ref{pres_2} below) of a torus knot group. In the final section of the paper, we construct an analogous Garside structure for the complex braid group of $G_{13}$. Using the fact, established by Bannai~\cite{Bannai}, that this braid group is isomorphic to the Artin group of dihedral type $I_2(6)$ or $G_2$, we can quite surprisingly construct a new Garside structure for all even dihedral Artin groups, similar to the one obtained for $G_{13}$ and torus knot groups. It would be interesting to understand the exact framework in which the Garside structure constructed in this paper can be built. Note that in~\cite{Gobet_toric}, we developed a framework to associate to any torus knot group a quotient group, which is infinite in general, and behaves like a complex reflection group of rank two. This could be a framework to try to realize torus knot groups as so-called interval groups (see~\cite[Theorem 0.5.2]{Dual}) in an interesting way, although these quotients are infinite in general.     

\medskip

The paper is organized as follows. In Section~\ref{facts_garside} we collect general results from the theory of Garside monoids and groups which we will apply later to show that our newly introduced presentation is the presentation of a Garside monoid and group. Section~\ref{sec_gar_tor} reviews existing results on Garside structures on torus knot groups. In Section~\ref{sec_new_pres}, we introduce our new presentations, show that they are presentations of torus knot groups, and enlarge them into two other presentations with more (redundant) relations which will be required to show left- and right-cancellativity in Section~\ref{sec_cancellativity}. In Section~\ref{sec_garside_structure}, we show our main result, namely that the newly introduced presentations are Garside presentations (see Theorem~\ref{main_text}). In Section~\ref{sec_groups_analogous} we construct and analogous Garside structure for the complex braid group of $G_{13}$ and for dihedral Artin groups of even type.

\medskip 
\textbf{Acknowledgments}. Part of this paper was written in February 2022 while the author was attending the semester program \textit{Braids} at ICERM, Providence. He thanks ICERM and the organizers for the invitation and financial support.

\section{Facts from Garside theory}\label{facts_garside}

In this section, we collect a few basic facts about Garside monoids and groups, for later use. We mostly adopt the definitions and conventions from~\cite{Garside}. Note that, while \textit{loc. cit.} introduces most of the results used in this paper in the general framework of \textit{Garside categories}, we will only need them in the case of presented monoids, and therefore reproduce them here in this less general context for the comfort of the reader.

Every monoid has a unit element $1$. Let $M$ be a monoid. 

\begin{definition}[Divisors and multiples]
	Let $a,b,c\in M$. If $ab=c$ holds, we say that $a$ is a \defn{left-divisor} (respectively, that $b$ is a \defn{right-divisor}) of $c$ and that $c$ is a \defn{right-multiple} of $a$ (respectively a \defn{left-multiple} of $b$).
\end{definition}

\begin{definition}[Cancellativity]
	We say that $M$ is \defn{left-cancellative} (respectively \defn{right-cancellative}) if for all $a,b,c\in M$, the equality $ab=ac$ (resp. $ba=ca$) implies $b=c$. If~$M$ is both left- and right-cancellative then we simply say that $M$ is \defn{cancellative}. 
\end{definition}

\begin{theorem}[Ore's Theorem]\label{thm:ore}
	If $M$ is cancellative, and if any two elements $a,b\in M$ admit a common left-multiple, that is, if there is $c\in M$ satisfying $a'a=c=b'b$ for some $a',b'\in M$, then $M$ admits a group of fractions $G(M)$ in which it embeds. Moreover, if $\langle \S \ \vert \ \R \rangle$ is a presentation of the monoid $M$, then $\langle \S \ \vert \ \R \rangle$ is a presentation of $G(M)$. 
\end{theorem}

A proof of this theorem can be found for instance in~\cite[Section~1.10]{CP}. It is straighforward to prove the following (a proof can also be found in~\cite[Lemme 1.1]{dehornoy_garside}):

\begin{lemma}\label{lem_partial}
	If $M$ is left-cancellative (respectively right-cancellative) and $1$ is the only invertible element in $M$, then the left-divisibility (resp. right-divisibility) relation on~$M$ is a partial order. 
\end{lemma}

\begin{comment}
	\begin{proof}
		Reflexivity is clear as $M$ has a unit $1$ and transitivity is also clear (and both hold without the cancellativity assumption and without the assumption on invertible elements). Let $a,b\in M$ such that $a$ left-divides $b$ and $b$ left-divides $a$. Then there are $c,c'\in M$ satisfying $ac=b$ and $bc'=a$. Hence we get $b=ac=bc'c$. By left-cancellativity this implies $c'c=1$, hence $c=1=c'$ as $1$ is the only invertible element in $M$. Hence $a=b$, and the left-divisibility relation is reflexive. The proof of the right counterparts is similar. 
	\end{proof}
\end{comment}

\begin{definition}[Noetherian divisibility]\label{def:noeth}
	We say that the divisibility in $M$ is \defn{Noetherian} if there exists a function $\lambda: M \rightarrow \mathbb{Z}_{\geq 0}$ such that $$\forall a,b\in M, \lambda(ab)\geq \lambda(a)+\lambda(b)~\text{and}~a\neq 1 \Rightarrow \lambda(a)\neq 0.$$ We say that $M$ is \defn{right-Noetherian} (respectively \defn{left-Noetherian}) if every strictly increasing sequence of divisors with respect to left-divisibility (resp. right-divisibility) is finite. Note that if the divisibility in $M$ is Noetherian, then $M$ is both left- and right-Noetherian.
\end{definition}

Note that Noetherian divisibility implies that the only invertible element in $M$ is $1$ and that $M$ is infinite whenever~$M\neq \{1\}$. In particular, by Lemma~\ref{lem_partial}, in a cancellative monoid $M$ with Noetherian divisibility, both left-divisiblity and right-divisibility induce a partial order on~$M$.

\begin{definition}[Garside monoid]\label{def_garside}
	A \defn{Garside monoid} is a pair~$(M, \Delta)$ where $M$ is a monoid and $\Delta$ is an element of $M$, satisfying the following five conditions:
	\begin{enumerate}
		\item $M$ is left- and right-cancellative,
		\item the divisibility in $M$ is Noetherian,
		\item any two elements in $M$ admit a left- and right-lcm, and a left- and right-gcd,
		\item the left- and right-divisors of the element $\Delta$ coincide and generate $M$,
		\item the set of (left- or right-)divisors of $\Delta$ is finite.
	\end{enumerate}
\end{definition}

Note that under these assumptions, the restrictions of left- and right-divisibility to the set of divisors of $\Delta$ yield two lattice structures on this set. 

In general, checking the above five conditions is a nontrivial task, especially for the left- and right-cancellativity. But these conditions have strong implications. For instance, every Garside group has a solvable word problem, and is torsion-free. We refer the reader to~\cite{Garside} for more on the topic. 

Let $M$ be a Garside monoid. Firstly, by Ore's Theorem, we have that $M$ embeds into its group of fractions $G(M)$.

\begin{definition}[Garside group]
	A group $G$ is a \defn{Garside group} if $G\cong G(M)$ holds for some Garside monoid $M$. 
\end{definition} 

We now recall some known cancellativity criteria for presented monoids which will be used in Section~\ref{sec_cancellativity}. We recall them from~\cite[Section~II.4]{Garside} (extending results from~\cite{dehornoy_garside}; see also~\cite{dehornoy_monoids} for a more general criterion). Most of the definitions given in this section are borrowed from~\cite{Garside}. 

Assume that $M$ is a monoid defined by a presentation $\langle \S \ \vert \ \R \rangle$, where $\S$ is a finite set of generators and $\R$ a set of relations between words in $\S^*$, \ie, words with letters in the generating set $\S$.

\begin{definition}[Right-complemented presentation]
	The presentation $\langle \S \ \vert \ \R \rangle$ is \defn{right-complemented} if $\R$ contains no relation where one side is equal to the empty word, no relation of the form $s\cdots = s\cdots$ with $s\in\mathcal{S}$, and if for $s\neq t\in\mathcal{S}$, there is at most one relation of the form $s\cdots = t\cdots$ in $\mathcal{R}$. 
\end{definition}

Given a right-complemented presentation~$\langle \S \ \vert \ \R \rangle$ of a monoid~$M$, there is a uniquely determined partial map~$\theta: \mathcal{S}\times \mathcal{S}\longrightarrow \mathcal{S}^*$ such that $\theta(s, s)=1$ holds for all $s\in\mathcal{S}$ and such that for $s\neq t\in\mathcal{S}$, the words $\theta(s,t)$ and $\theta(t,s)$ are defined whenever there is a relation $s\cdots = t\cdots$ in $\mathcal{R}$, and are such that this relation is given by $s \theta(s,t) = t \theta(t,s)$. The map~$\theta$ is the \defn{syntactic right-complement} attached to the right-complemented presentation~$\langle \S \ \vert \ \R \rangle$. 

If $\langle\S \ \vert \ \R\rangle$ is right-complemented, then by~\cite[Lemma~II.4.6]{Garside}, the map~$\theta$ admits a unique minimal extension to a partial map from~$\mathcal{S}^* \times \mathcal{S}^*$ to~$\mathcal{S}^*$, still denoted $\theta$, and satisfying \begin{eqnarray}
	\label{c}\theta(s,s)=1,~\forall s\in\mathcal{S},\\
	\theta(bc,a)=\theta(c, \theta(b,a)),~\forall a,b,c\in\mathcal{S}^*, \\
	\theta(a,bc)=\theta(a,b)\theta(\theta(b,a),c),~\forall a,b,c\in\mathcal{S}^*,\\
	\label{d}\theta(1,a)=a\text{ and }\theta(a,1)=1,~\forall a\in\mathcal{S}^*.
\end{eqnarray}

We illustrate some of these relations in the diagram in Figure~\ref{reltheta}. 

\begin{figure}[h!]
	\begin{center}
		\begin{tikzpicture}
			\matrix (m) [matrix of math nodes,row sep=3em,column sep=4em,minimum width=2em]
			{
				\bullet & \bullet & \bullet \\
				\bullet & \bullet & \bullet \\};
			\path[-stealth]
			(m-1-1) edge node [left] {$a$} (m-2-1)
			edge node [below] {$b$} (m-1-2)
			(m-2-1.east|-m-2-2) edge node [below] {$\theta(a,b)$} (m-2-2)
			%node [above] {$\exists$} (m-2-2)
			(m-1-2) edge node [right] {$\theta(b,a)$} (m-2-2)
			(m-1-2) edge node [below] {$c$} (m-1-3)
			(m-2-2.east|-m-2-3) edge node [below] {$\theta(\theta(b,a),c)$} (m-2-3)
			(m-1-3) edge node [right] {$\theta(c, \theta(b,a))$} (m-2-3);
		\end{tikzpicture}
	\end{center}
	\caption{Commutative diagram illustrating the relations $\theta(bc,a)=\theta(c, \theta(b,a))$ and $\theta(a,bc)=\theta(a,b)\theta(\theta(b,a),c)$. Arrows represent elements of the monoid and composition of arrows corresponds to the product in $M^{\mathrm{op}}$.}
	\label{reltheta}
\end{figure}

\begin{definition}[Cube condition]\label{def_cube} Given a right-complemented presentation~$\langle \S \ \vert \ \R \rangle$ of a monoid~$M$ with syntactic right-complement~$\theta$, we say that the \defn{$\theta$-cube condition holds} (respectively that the \defn{sharp $\theta$-cube condition holds}) for a triple~$(a,b,c)\in({\mathcal{S}^*})^3$ if either both $\theta(\theta(a,b), \theta(a,c))$ and $\theta(\theta(b,a), \theta(b,c))$ are defined and represent words in~$\mathcal{S}^*$ that are equivalent under the set of relations~$\mathcal{R}$ (resp. that are equal as words), or neither of them is defined. 
\end{definition} 

\begin{definition}[Conditional lcm]
	We say that a left-cancellative (respectively right-cancellative) monoid~$M$ with no nontrivial invertible element \defn{admits conditional right-lcms} (resp. \defn{admits conditional left-lcms}) if any two elements of $M$ admitting a common right-multiple (resp. a common left-multiple) admit a common right-lcm (resp. a common left-lcm). 
\end{definition}

\begin{prop}[{see~\cite[Proposition~II.4.16]{Garside}}]\label{cancellative_criterion}
	If $\langle \S \ \vert \ \R \rangle$ is a right-complemented presentation of a monoid~$M$ with syntactic right-complement~$\theta$, and if $M$ is right-Noetherian and the $\theta$-cube condition holds for every triple of pairwise distinct elements of $\mathcal{S}$, then $M$ is left-cancellative, and admits conditional right-lcms. More precisely, $a$ and $b$ admit a common right-multiple if and only if $\theta(a,b)$ exists and, then, $a\theta(a,b)=b\theta(b,a)$ represents the right-lcm of these elements.  
\end{prop}

For later use we also state the following result:

\begin{lemma}[{see~\cite[Lemma~II.2.22]{Garside}}]\label{lemma_gcd}
	If $M$ is cancellative and admits conditional right-lcms (respectively conditional left-lcms), then any two elements of $M$ that admit a common left-multiple (resp. a common right-multiple) admit a right-gcd (resp. a left-gcd).  
\end{lemma}

This section is devoted on recalling the definition and a few properties of Garside elements. A proof of the following lemma can be found in~\cite[Lemme 1.8]{dehornoy_garside}.

\begin{lemma}\label{cond_delta_lcm}
	Let $M$ be a cancellative monoid with no nontrivial invertible element (so that left- and right-divisibility relations are partial orders on $M$). Assume that $M$ has conditional (left- and right-) lcms, and that $M$ has an element~$\Delta$ satisfying the following assumptions \begin{itemize}
		\item the sets of left- and right-divisors of $M$ coincide, and form a finite set,
		\item the set of divisors of $\Delta$ generate $M$.
	\end{itemize}
	Then any two elements $x, y\in M$ admit a left-lcm and a right-lcm.
\end{lemma}

\begin{comment}
	\begin{proof}
		As $M$ has conditional lcms, it suffices to show that any two elements~$x,y\in M$ have a (left- or right-)common multiple. We show that $x,y$ have a common right-multiple (the proof for left-multiples is similar). Note that under our assumptions, if $z$ is any divisor of $\Delta$, then $z \Delta=\Delta z'$ for some divisor~$z'$ of $\Delta$.  
		
		Let $x=x_1 x_2\cdots x_k$, $y=y_1 y_2\cdots y_\ell$ where the $x_i$'s and $y_i$'s are divisors of $\Delta$. Without loss of generality we can assume that $\ell \leq k$ holds. Then we claim that $\Delta^k$ is a common right-multiple of $x$ and $y$. To this end, it suffices to show that if $a=a_1 a_2 \cdots a_m$ is an element of $M$ which is a product of $m$ divisors~$a_i$ of $\Delta$, then $a$ is a left-divisor of $\Delta^m$. As $a_1$ is a left-divisor of $\Delta$ and left and right-divisors of $\Delta$ coincide, we can write $\Delta^m= a_1 \Delta^{m-1} b_1$, where $b_1$ is a divisor of $\Delta$. Iterating, we eventually end up with a decomposition~$\Delta^m= a_1 a_2 \cdots a_m b_m \cdots b_2 b_1$, which concludes the proof. 
	\end{proof}
\end{comment}

\begin{definition}[Garside element]
	If $M$ and $\Delta$ satisfy the assumptions of the above lemma, we say that $\Delta$ is a \defn{Garside element} in $M$. In this case we denote by $\Div(\Delta)$ the set of left-divisors of $\Delta$ (which is equal to the set of right-divisors of $\Delta$). We call its elements the \defn{simples} of $(M,\Delta)$. 
\end{definition}

\section{Garside structures on torus knot groups}\label{sec_gar_tor}

Let $n,m\geq 2$, with $n$ and $m$ coprime. The \defn{$(n,m)$-torus knot group} $G(n,m)$ is the knot group of the torus knot $T_{n,m}$ (see~\cite{Rolfsen}). As $T_{n,m}$ and $T_{m,n}$ are isotopic, we have $G(n,m)\cong G(m,n)$. There is a well-known presentation of $G(n,m)$ given by \begin{equation}\label{pres_1}\langle \ x, y \ \vert\ x^n=y^m
\ \rangle.\end{equation}
It was shown by Schreier~\cite{Schreier} that the center of $G(n,m)$ is infinite cyclic, generated by $x^n=y^m$. Another presentation of $G(n,m)$ is given by \begin{equation}\label{pres_2}\langle \ x_1, x_2, \dots, x_n \ \vert\ 
\underbrace{x_1 x_2 \cdots}_{m~\text{factors}} = \underbrace{x_2 x_3\cdots}_{m~\text{factors}} = \dots = \underbrace{x_n x_1 \cdots}_{m~\text{factors}}
\ \rangle,\end{equation} where indices are taken modulo $n$ if $n<m$.

Since $G(n,m)\cong G(m,n)$, a third presentation is given by \begin{equation}\label{pres_3}\langle \ y_1, y_2, \dots, y_m \ \vert\ 
\underbrace{y_1 y_2 \cdots}_{n~\text{factors}} = \underbrace{y_2 y_3\cdots}_{n~\text{factors}} = \dots = \underbrace{y_m y_1 \cdots}_{n~\text{factors}}
\ \rangle.\end{equation}

For $n=2$ and $m=3$, Presentation~\ref{pres_2} is nothing but the standard presentation of the $3$-strand braid group $\mathcal{B}_3$, while Presentation~\ref{pres_3} is its Birman-Ko-Lee (\cite{BKL}) or dual (\cite{Dual}) presentation. 

Note that, on the algebraic side, it is not obvious that Presentations~\ref{pres_1}, \ref{pres_2}, and \ref{pres_3} define isomorphic groups. The link between these presentations is given in the following Lemma, which is straightforward to check. 

\begin{lemma}
Assume that $n<m$. \begin{enumerate} \item The map $$y_1\mapsto x_1\text{, and for }2\leq i \leq m, y_i\mapsto x_n^{-1} x_{n-1}^{-1} \cdots x_{n+3-i}^{-1} x_{n+2-i} x_{n+3-i}\cdots x_n$$ (where indices in the $x_i$'s are taken modulo $n$) defines an isomorphism between the group with presentation~\ref{pres_3} and the group with presentation~\ref{pres_2}.  
\item The map $$x\mapsto x_1 x_2 \cdots x_m, \ y\mapsto x_2 x_3\cdots x_n x_1$$ (where indices in the $x_i$'s are taken modulo $n$) defines an isomorphism between the group with presentation~\ref{pres_1} and the group with presentation~\ref{pres_2}. 
\end{enumerate}
\end{lemma}

Torus knot groups are examples of groups which possess many non-isomorphic Garside monoids. Indeed, it was shown by Dehornoy and Paris that Presentations~\ref{pres_1}, \ref{pres_2} and \ref{pres_3} define Garside monoids (see~\cite[Examples 4 and 5]{DP}). Picantin~\cite{Picantin_torus} gave Garside presentations for all torus link groups, yielding in the particular case of knots a Garside presentation which is similar to~\ref{pres_2}, but distinct in general. In the specific case of $(n,n+1)$-torus knot groups ($n\geq 2$), a new Garside presentation was given by the author in~\cite{Gobet}, generalizing the exotic presentation $\langle \ a, b \ \vert \ aba=b^2 \ \rangle$ of the $3$-strand braid group $\mathcal{B}_3\cong G(2,3)$. We shall generalize this Garside structure to all torus knot groups in the next sections. 

\begin{exple}\label{exple_g12}
We list several non-isomorphic Garside monoids for $G(3,4)$, which is also isomorphic to the braid group of the exceptional complex reflection group $G_{12}$ (see~\cite{Bannai}). Presentations~\ref{pres_1} to~\ref{pres_3} respectively yield the Garside presentations $$\langle \ x, y \ \vert \ x^2= y^3 \ \rangle, ~\bigg\langle x_1, x_2, x_3 \ \bigg\vert\ 
\begin{matrix}
	x_1 x_2 x_3 x_1\\ = x_2 x_3 x_1 x_2\\
	=x_3 x_1 x_2 x_3
\end{matrix}
\ \bigg\rangle, ~\bigg\langle y_1, y_2, y_3, y_4 \ \bigg\vert\ 
\begin{matrix}
	y_1 y_2 y_3 =y_2 y_3 y_4\\ = y_3 y_4 y_1=y_4 y_1 y_2
\end{matrix}
\ \bigg\rangle,$$ where the Garside element is given respectively by $x^2$, $x_1 x_2 x_3 x_1$, and $y_1 y_2 y_3$. Picantin's presentation~\cite[Lemma 3.2]{Picantin_torus} yields the presentation $$\langle \ \sigma_1,\sigma_2,\sigma_3 \ \vert\
\sigma_1 \sigma_2 \sigma_3 \sigma_1=\sigma_2 \sigma_1 \sigma_2 \sigma_3,~\sigma_3 \sigma_1 \sigma_2 \sigma_3=\sigma_1 \sigma_2 \sigma_3 \sigma_2
\ \rangle,$$ which also appears in work of Bessis-Bonnafé-Rouquier~\cite{BBR}. The Garside element is given by $(\sigma_1 \sigma_2 \sigma_3)^4$. Note that Picantin also exhibited alternative Garside monoids for $G(3,4)$, including the monoid $$\langle \ x, y \ \vert \ xyxyxyx=y^2 \ \rangle$$ which has Garside element $y^3$, and even an infinite family of Garside monoids (see~\cite[Remark 5.2]{Picantin_torus}). The presentation obtained in~\cite{Gobet} is given by $$\langle \ \rho_1, \rho_2, \rho_3 \ \vert\ 
\rho_1 \rho_3 \rho_1 = \rho_2 \rho_3,~
\rho_1 \rho_3 \rho_2=\rho_3^2 \ \rangle$$ and the Garside element is given by $\rho_3^4$. \end{exple} 

\section{New presentations}\label{sec_new_pres}

Let $n, m$ be two integers such that $2 \leq n$, $n<m$ and $n$ and $m$ are coprime. Let $m=qn+r$ be the Euclidean division of $m$ by $n$ ; in particular $0<r<n$. Consider the monoid $\mathcal{M}(n,m)$ defined by the presentation $\langle \mathcal{S} \ \vert \ \mathcal{R} \rangle$ given by 
\begin{equation}\label{pres_brut}
\bigg\langle \omega_1, \omega_2, \dots, \omega_n \ \bigg\vert\ 
\begin{matrix}
\omega_r  \omega_n^q \omega_{i-r} =\omega_i \omega_n^q~\text{if }r < i \leq n,\\
\omega_r  \omega_n^q \omega_{n+i-r}=\omega_i \omega_n^{q+1}~\text{if }1\leq i <r.
\end{matrix}
\ \bigg\rangle
\end{equation}
We will denote by $\mathcal{G}(n,m)$ the group defined by the same presentation. 
\begin{lemma}\label{noeth_new}
The map $\lambda :\{\omega_1, \omega_2, \dots, \omega_{n}\}\longrightarrow \mathbb{Z}_{\geq 0}$, $\omega_i\mapsto i$ extends to a uniquely defined length function $\lambda$ on $\mathcal{M}(n,m)$ satisfying $\lambda(ab)=\lambda(a)+\lambda(b)$ for all $a,b\in \mathcal{M}(n,m)$. In particular, the divisibility in $\mathcal{M}(n,m)$ is Noetherian, and $\mathcal{M}(n,m)$ is both left- and right-Noetherian.  
\end{lemma}

\begin{proof}
It suffices to check that the extension of $\lambda$ to $\S^*$ takes the same value on each side of any relation in $\R$, which is immediate.  
\end{proof}

\begin{lemma}\label{lem_red_garside}
We have $(\omega_r \omega_n^q)^{n-1}\omega_r = \omega_n^{m-q}$. Hence $$\omega_n^m=(\omega_r \omega_n^q)^n=(\omega_n^q\omega_r )^n.$$
\end{lemma}
\begin{proof}
For $1\leq i < n$ we have $\omega_r \omega_n^q \omega_i=\omega_{i+r} \omega_n^q$ if $i+r \leq n$ and $\omega_{i+r-n} \omega_n^{q+1}$ otherwise. In fact, considering indices modulo $n$, we have $\omega_r \omega_n^q \omega_i=\omega_{i+r} \omega_n^m$, where $m$ is equal to either $q$ or $q+1$; the value of $m$ can be recovered using the fact that the relations are homogeneous. This means that, starting from the word $(\omega_r \omega_n^q)^{n-1}\omega_r$, the rightmost factor $\omega_r$ can be moved to the left by a successive application of the relations, namely, for suitable exponents $m_j$ we have
\begin{align*}
(\omega_r \omega_n^q)^{n-1}\omega_r &=(\omega_r \omega_n^q)^{n-2}\omega_{2r} \omega_n^{m_1}=(\omega_r \omega_n^q)^{n-3}\omega_{3r} \omega_n^{m_2}=\dots=(\omega_r \omega_n^q)^{n-j-1}\omega_{(j+1)r} \omega_n^{m_j}\\ &=\dots=\omega_{nr} \omega_n^{m_{n-1}}. 
\end{align*}
As indices are considered modulo $n$ we have $\omega_{nr}=\omega_n$ while $\omega_{kr}\neq \omega_n$ for $k<n$ since $n$ and $r$ are coprime. By homogeneity of the defining relations, since the word from which we started has length $m(n-1)+r$, we have $m_{n-1}=m-q-1$. Hence $(\omega_r \omega_n^q)^{n-1}\omega_r=\omega_{nr} \omega_n^{m_{n-1}}=\omega_n^{m-q},$ which concludes the proof. 
\end{proof}

The proof of Lemma~\ref{lem_red_garside} suggests to reindex the $\omega_i$'s as follows. As $n$ and $m$ are coprime $r$ and $n$ are also coprime. If $\alpha$ is a generator of the cyclic group $C_n$ of order $n$, since $n$ and $r$ are coprime we have that $\alpha^r$ is also a generator of $C_n$. Hence considering residues modulo $n$, we have $\{ \overline{r}, \overline{2r}, \overline{3r}, \dots, \overline{nr} \}= \mathbb{Z}/n\mathbb{Z}$. For $1\leq i \leq n$, define $k_i$ as the unique integer such that $1\leq k_i \leq n$ and $\overline{ir}=\overline{k_i}$. It follows from the above observation that $\{k_1, k_2, \dots, k_n\}=\{1, 2, \dots, n\}$. Note that $k_1=r$, $k_n=n$, and $k_{n-1}=n-r$.

 Setting $\rho_i:=\omega_{k_i}$ the defining relations of $\mathcal{M}(n,m)$ or $\mathcal{G}(n,m)$ can be rewritten \begin{align}\label{rel_ref}\rho_{1} \rho_n^q \rho_{i}=\rho_{i+1} \rho_n^{\frac{m+k_i-k_{i+1}}{n}},~\forall 1\leq i \leq n-1.
\end{align} Note that $k_{i+1}-k_i=r$ or $-n+r$. \begin{definition} We define the \defn{defect} $D(i)$ of $i$ as \begin{align}\label{def_defect}D(i):=\frac{r+k_i - k_{i+1}}{n}=\left\{\begin{array}{ll}
1  & \mbox{if } k_i+r>n, \\
0 & \mbox{otherwise},
\end{array}\right.\end{align}
Note that the indices in the $k_{j}$'s can be viewed modulo $n$, and hence $D(i)$ can be defined for $i\in\mathbb{Z}$. We have $D(n-1)=0$ and $D(n)=1$. We say that $i$ is \defn{good} if $D(i)=0$. Otherwise we say that $i$ is \defn{bad}.  
\end{definition} We thus get that the monoid $\mathcal{M}(n,m)$ admits the presentation 
\begin{equation}\label{pres_pratique}
	\bigg\langle \rho_1, \rho_2, \dots, \rho_n \ \bigg\vert\ 
\rho_{1} \rho_n^q \rho_{i}=\rho_{i+1} \rho_n^{q+D(i)}, ~\forall 1\leq i \leq n-1 
	\ \bigg\rangle.
\end{equation}

\begin{exple}\label{exple_old}
	In the particular case where $m=n+1$, we get $r=1$, $q=1$, $k_i=i$ for all $1\leq i\leq n$ and $i$ is good for all $i \neq n$, and the monoid $\mathcal{M}(n,n+1)$ is simply given by the presentation
	\begin{equation}\label{old}
		\bigg\langle \rho_1, \rho_2, \dots, \rho_n \ \bigg\vert\ \rho_1  \rho_n \rho_{i} =\rho_{i+1} \rho_n,~\forall 1 \leq i \leq n-1
		\ \bigg\rangle
	\end{equation}

Note that in this specific case we exactly recover the monoid studied in~\cite{Gobet}.
\end{exple}

\begin{exple}\label{pres_3_5} For $n=3$ and $m=5$, we have $q=1$, $r=2$ $k_1=2, k_2=1, k_3=3$, hence $D(1)=1$ and $D(2)=0$. we get the presentation $$\langle \  \rho_1, \rho_2, \rho_3  \ \vert \ \rho_1 \rho_3 \rho_1=\rho_2 \rho_3^2,~\rho_1 \rho_3 \rho_2=\rho_3^2 \ \rangle$$ This is in fact a presentation of $G(3,5)$ (which by~\cite[Theorem 1 (ix)]{Bannai} is also isomorphic to the braid group of the complex reflection group $G_{22}$). This is a general fact, proven in Proposition~\ref{isom_torus} below. 
\end{exple}

Let us introduce some further notation which will be helpful in proofs. Let $1\leq i \leq j \leq n$. Consider the sequence $S(i,j)$ of pairs \begin{align}\label{seq_sij}(k_i, k_j) \rightarrow (k_{i+1}, k_{j+1}) \rightarrow \cdots \rightarrow (k_{i+n-j-1}, k_{n-1}).\end{align}
This sequence contains $n-j$ pairs (if $n=j$ then by convention it is empty). Let \begin{align*}B(i,j):=\#\{ \ell\in \{i, i+1, \dots, i+n-j-1\}~|~\ell~\text{is bad}\}\end{align*}
In the sequence~\eqref{seq_sij}, the integer $B(i,j)$ is the number of bad indices in first position of the pairs, while $B(j,j)$ is the number of bad indices in the second position. Note that, for $1\leq i \leq j \leq k \leq n$, since $\{i, i+1, \dots, k-1\}=\{i, i+1, \dots, j-1\}\coprod \{j, j+1, \dots, k-1\}$ we have \begin{align}\label{sumofbij} B(i,i+n-j)+B(j,j+n-k)=B(i,i+n-k).
\end{align} 
\begin{definition} We define the \defn{defect} of the sequence $S(i,j)$ by $D(i,j):=B(i,j)-B(j,j)$. Note that \begin{align*}D(i,j)=\left(\sum_{\ell=i}^{i+n-j-1} D(\ell)\right)-\left(\sum_{\ell=j}^{n-1} D(\ell)\right)=\frac{(n-j)r +k_i - k_{i+n-j}}{n} - \frac{(n-j)r + k_j - k_{n}}{n}
\end{align*} yielding \begin{align}\label{dij_alt_def} D(i,j)=\frac{n + k_i - k_j - k_{i+n-j}}{n}.
\end{align}\end{definition}
In other words, the integer $D(i,j)$ is the difference between the number of bad indices in first position of pairs of $S(i,j)$, and the number of bad indices in second position of pairs. We will not require the sequence $S(i,j)$ in our proofs but it is useful to have it in mind to have an interpretation of $D(i,j)$. Most of the time we will use~\eqref{dij_alt_def} as definition for $D(i,j)$, but sometimes we will also use the definition as $B(i,j)-B(j,j)$. In some cases we may require to use $D(i,j)$ even if $i$ or $j$ is not in $\{1,2,\dots, n\}$ in which case we will always use~\eqref{dij_alt_def} as definition again viewing the indices in $k_j$'s modulo $n$.

Since $k_\ell\in\{1,2,\dots, n\}$ for all $\ell$, we have $-2n < k_i - k_j - k_{i+n-j} < n$, and since $D(i,j)$ is an integer we get: 
\begin{lemma}\label{lem_01}
	We have $D(i,j)\in\{0,1\}$. 
\end{lemma}

We will often use this fact, sometimes without referring to the above lemma. Also note that $D(i,i+1)=D(i)$ for all $i\in\mathbb{Z}$. 

The following lemma will be helpful to show left-cancellativity of $\mathcal{M}(n,m)$.

\begin{lemma}\label{lem_rel_left}
	\begin{enumerate} \item Let $1\leq i < j \leq n$. We have $(\rho_1 \rho_n^q)^i \rho_{{j-i}}=\rho_{j} \rho_n^{qi+ B(j-i, n-i)}.$
		\item Let $1\leq i \leq n$. We have $(\rho_1 \rho_n^q)^i =\rho_{i} \rho_n^{qi+B(1,n-i+1)}=\rho_{i} \rho_n^{\frac{mi-k_i}{n}}$.\end{enumerate}
\end{lemma} 

\begin{proof}
	The first point is just the result of a successive application of the defining relations. Indeed, for $i=1$ this is just a defining relation as in~\eqref{pres_pratique} (note that $B(j-1,n-1)=D(j-1)$) and for $i>1$ one just moves the $\rho_{{j-i}}$ to the left as we did in the proof of Lemma~\ref{lem_red_garside}, applying $i$ times a defining relation and using $\sum_{\ell=j-i}^{j-1} D(\ell)=B(j-i,n-i)$. The second point is obtained from the first one as $(\rho_1 \rho_n^q)^{i}=(\rho_1 \rho_n^q)^{i-1} \rho_1 \rho_n^q=\rho_i \rho_n^{q(i-1)+B(1,n-i+1)+q}$ and as $k_1=r$ we have $B(1,n-i+1)=\sum_{\ell=1}^{i-1} D(\ell)=\frac{ir-k_i}{n}.$ 
\end{proof}

\begin{cor}\label{cor_pres_left_c}
	The monoid $\mathcal{M}(n,m)$ and the group $\mathcal{G}(n,m)$ have a presentation with the same set of generators $\mathcal{S}=\{\rho_1, \rho_2, \dots, \rho_n\}$ as before and relations $\mathcal{R}'$ given by  \begin{equation}\label{red_first}\rho_{i} \rho_n^{qi+B(1,n-i+1)} \rho_{{j-i}}=\rho_{j} \rho_n^{qi+B(j-i,n-i)},~\forall 1\leq i < j \leq n.
	\end{equation}
\end{cor}
\begin{proof} This is an immediate consequence of the previous lemma, together with the fact that this new set of relations contains $\mathcal{R}$ (set $i=1$). \end{proof}

The advantage of this new (and redundant) presentation $\langle \S \ \vert \ \R'\rangle$ is that it is still right-complemented, with syntactic right-complement $\theta$ given for $1\leq i < j \leq n$ by \begin{align}\label{syntactic_right}\theta(\rho_{i}, \rho_{j})= \rho_n^{qi+B(1,n-i+1)} \rho_{{j-i}}, ~\theta(\rho_{j}, \rho_{i})=\rho_n^{qi+B(j-i,n-i)}.
\end{align}

We will use this presentation to show that $\mathcal{M}(n,m)$ is left-cancellative by showing in Lemma~\ref{cube_left} below that it satisfies the sharp $\theta$-cube condition (which fails if applied to the presentation~$\langle \mathcal{S} \ \vert \ \mathcal{R} \rangle$).

A little more work is required to get a suitable presentation to show right-cancellativity.

\begin{prop}\label{rel_droite}
	Let $1\leq i<j \leq n$. \begin{enumerate}
		\item If $D(i,j)=1$, then $(\rho_1 \rho_n^q)^{n-j} \rho_{i}=\rho_{{n-j+i}}(\rho_1 \rho_n^q)^{n-j} \rho_{j}.$
		\item If $D(i,j)=0$, then $(\rho_1 \rho_n^q)^{n-j+1} \rho_{i}=
		\rho_{{n-j+i+1}} \rho_{n}^{q-1+D(n-j+i)} (\rho_1 \rho_n^q)^{n-j} \rho_{j}.$
	\end{enumerate}
\end{prop}

\begin{proof}
	By Lemma~\ref{lem_rel_left}~(1) we have
	$(\rho_1 \rho_n^q)^{n-j} \rho_{i}=\rho_{{i+n-j}} \rho_n^{q(n-j)+B(i,j)}.$ On the other hand, also by Lemma~\ref{lem_rel_left}~(1) we get $$\rho_{{n-j+i}}(\rho_1 \rho_n^q)^{n-j}\rho_{j}=\rho_{{n-j+i}}\rho_{n} \rho_n^{q(n-j)+B(j,j)}=\rho_{{n-j+i}}\rho_n^{q(n-j)+B(j,j)+1}.$$
	
	If $D(i,j)=1$, then using that $D(i,j)=B(i,j)-B(j,j)$ we get the claimed relation. 
	
	Now assume that $D(i,j)=0$. Since as we saw above we have $(\rho_1 \rho_n^q)^{n-j} \rho_{i}=\rho_{{i+n-j}} \rho_n^{q(n-j)+B(i,j)}$ we obtain \begin{align*}(\rho_1 \rho_n^q)^{n-j+1} \rho_{i}&=\rho_1 \rho_n^q\rho_{{i+n-j}} \rho_n^{q(n-j)+B(i,j)}=\rho_{{n-j+i+1}} \rho_n^{q(n-j+1)+B(i,j)+D(n-j+i)}\\ &=\rho_{{n-j+i+1}} \rho_n^{q(n-j+1)+B(j,j)+D(n-j+i)}=\rho_{{n-j+i+1}}\rho_n^{q-1+D(n-j+i)} \rho_n\rho_n^{q(n-j)+B(j,j)} \\ &=\rho_{{n-j+i+1}}\rho_n^{q-1+D(n-j+i)} (\rho_1 \rho_n^q)^{n-j} \rho_j, 
		\end{align*} which concludes the proof. 
	\end{proof}
	
	For all $i,j\in\{1,2,\dots, n\}$ with $i<j$, Proposition~\ref{rel_droite} gives us exactly one relation of the form $\cdots \rho_{i}=\cdots \rho_{j},$ which we denote by $R(i,j)$. Let $\mathcal{R}''$ be the set of relations $R(i,j)$. We thus get: 
	
	\begin{cor}\label{cor_pres_right_c}
		The presentation $\langle \mathcal{S} \ \vert \ \mathcal{R}'' \rangle$ is a presentation of the monoid $\mathcal{M}(n,m)$ and the group $\mathcal{G}(n,m)$. 	
	\end{cor}
	
	\begin{proof}
		Note that when $j=n$, we always have $D(i,j)=0$ (the sequence $S(i,j)$ is empty in that case). For $1\leq i < n$, the relation $R(i,n)$ is therefore given by $\rho_1 \rho_n^q \rho_{i}=\rho_{{i+1}} \rho_n^{q+D(i)}$. But these are exactly the defining relations $\mathcal{R}$ of $\mathcal{M}(n,m)$. Hence the defining relations of $\mathcal{M}(n,m)$ are a subset of the relations $\mathcal{R}''$. Conversely, we have seen in Proposition~\ref{rel_droite} that all relations $R(i,j)$ are satisfied in $\mathcal{M}(n,m)$, hence that they are all consequences of the relations $\mathcal{R}$. 	
	\end{proof}	
	
	The above presentation will be helpful to show that $\mathcal{M}(n,m)$ is right-cancellative. This is equivalent to showing that the opposite monoid $\mathcal{M}(n,m)^{\mathrm{op}}$ is left-cancellative. Hence let $\mathcal{T}=\{\tau_{i}\}$ be a set in bijection with $\mathcal{S}=\{\rho_{i}\}$ ($\tau_i \leftrightarrow \rho_i$) and consider the set of relations $(\mathcal{R}'')^{\mathrm{op}}$ between these $\tau_{i}$'s given by reversing the relations $\mathcal{R}$. This yields a presentation $\langle \mathcal{T} \ \vert \ (\mathcal{R}'')^{\mathrm{op}}\rangle$ which is right-complemented. We denote by $\eta$ its syntactic right-complement, given for $1\leq i < j \leq n$ by 
		\begin{align}
			\label{right_less} \eta(\tau_{i}, \tau_{j})&=\left\{ \begin{array}{ll} (\tau_n^q \tau_1)^{n-j} & \mbox{if~}D(i,j)=1, \\ (\tau_n^q \tau_1)^{n-j+1} & \mbox{if~}D(i,j)=0,  \end{array}\right. \\
			\label{right_great} \eta(\tau_j, \tau_i)&=\left\{ \begin{array}{ll}(\tau_n^q \tau_1)^{n-j}\tau_{n-j+i} & \mbox{if~}D(i,j)=1,\\ (\tau_n^q \tau_1)^{n-j}\tau_n^{q-1+D(n-j+i)}\tau_{n-j+i+1} & \mbox{if~}D(i,j)=0.\end{array}\right. 
		\end{align} 

\begin{prop}[Isomorphism with torus knot groups]\label{isom_torus}
The group $\mathcal{G}(n,m)$ defined by Presentation~\ref{pres_pratique} is isomorphic to $G(n,m)$. An isomorphism is given in terms of the generators of Presentations~\ref{pres_1} and~\ref{pres_pratique} by $\rho_{i}\mapsto x^i y^{-qi-B(1,n-i+1)}$ for all $1\leq i \leq n$.
\end{prop}

\begin{proof}
We first claim that the map $\varphi$ defined on generators of $\mathcal{G}(n,m)$ by $\varphi(\rho_{i})= x^i y^{-qi-B(1,n-i+1)}$ for all $1\leq i \leq n$ extends to a group homomorphism $\mathcal{G}(n,m)\longrightarrow G(n,m)$. Note that $B(1,1)=r-1$ since $B(1,1)$ counts the number of bad indices in $\{1,2,\dots, n-1\}$, and such indices are $n-r+1, n-r+2, \dots, n-1$. We thus have $\varphi(\rho_n)=x^n y^{-qn-r+1}=y^m y^{-m} y = y$. We have $B(1,n)=0$ which yields $\varphi(\rho_1)=xy^{-q}$. We have \begin{align*} 
\varphi(\rho_1)  \varphi(\rho_n)^q \varphi(\rho_{i})&=x y^{-q} y^q x^i y^{-qi-B(1,n-i+1)}=x^{i+1} y^{-qi-B(1,n-i+1)}=x^{i+1} y^{-qi-B(1,n-i)+D(i)}\\ &=x^{i+1}y^{-q(i+1)-B(1,n-i)}y^{q+D(i)}=\varphi(\rho_{{i+1}})\varphi( \rho_n)^{q+D(i)}.
\end{align*}
Hence the relations of Presentation~\ref{pres_pratique} are satisfied, and $\varphi$ extends to a group homomorphism.
Conversely, we show that the map $\psi$ defined on generators of $G(n,m)$ by $\psi(y)=\rho_n$ and $\psi(x)=\rho_1 \rho_n^q$ extends to a group homomorphism from $G(n,m)$ to $\mathcal{G}(n,m)$. We have $\psi(y)^m=\rho_n^m$ while $\psi(x)^n=(\rho_1 \rho_n^q)^n$ which, by Lemma~\ref{lem_red_garside}, is equal to $\rho_n^m$. Hence $\psi$ extends to a group homomorphism. 

The fact that $\varphi\circ\psi=\mathrm{Id}$ is clear, and the fact that $\psi\circ\varphi=\mathrm{Id}$ follows from Lemma~\ref{lem_rel_left}.   
\end{proof}

For later use we prove a few elementary formulas involving defects:

\begin{lemma}\label{ft}
 
\begin{enumerate}
	\item Let $1\leq i < j < n$. If $D(i)\neq D(j)$, then $D(i+1,j+1)=D(j)$. 
	\item Let $1\leq i < j < \ell \leq n$. If $D(i,j)=D(j,\ell)$, then $D(i,\ell)=D(i,j)$.
	\item Let $1 \leq i < j < n$. Then $$ D(i+2, j+1)+D(i+1)=D(i+1,j)+D(j).$$
	\item Let $1\leq i < j < \ell\leq n$. If $D(j,\ell)=D(i,\ell)$, then $D(n-\ell+i, n-\ell+j)=D(i,j)$. 
	\item Let $1\leq i < j-1 \leq n-1$. If $D(i)=D(i,j)=0$ and  $D(i+1,j)=1$, then $D(n-j+i)=1$.
	\item Let $1\leq i < j \leq n$. If $D(i)=1$ and $D(j)=0$, then $D(n-j+i)=D(i+1,j)$.  
\end{enumerate}
\end{lemma}

\begin{proof}
	For the first point, using the definitions of $D(i)$, $D(j)$ and the formula for $D(i+1,j+1)$ given in~\eqref{dij_alt_def}, we get that $$n D(i+1,j+1)=n+k_i - k_j +n(D(j)-D(i)) - k_{i+n-j}.$$ If $D(j)=1$ and $D(i)=0$, this yields $n D(i+1,j+1)=2n + k_i - k_j - k_{i+n-j}$. But the right hand side has to be equal to either $0$ or $n$ as $D(i+1,j+1)\in\{0,1\}$, and the only possibility is that it is equal to $n$ as $k_i, k_j, k_{i+n-j}\in\{1,2,\dots, n\}$. We thus get $D(i+1,j+1)=1=D(j)$. Similarly if $D(j)=0$ and $D(i)=1$ one gets $D(i+1,j+1)=0$.
	
	For the second point, using~\eqref{dij_alt_def} again we get $$n D(i,\ell)=n(D(i,j)+D(j,\ell))+ k_{i+n-j} +k_{j+n-\ell} -n -k_{i+n-\ell}.$$ Distinguishing the cases as in the first point we easily get the claim.  
	
	The third point is the result of a direct computation using~\eqref{dij_alt_def} and~\eqref{def_defect}. 
	
	For the fourth point, the condition $D(j,\ell)=D(i,\ell)$ yields $k_{n-\ell+i}-k_{n-\ell+j}=k_i-k_j$. We then have \begin{align*}n D(n-\ell+i, n-\ell+j)&= n+k_{n-\ell+i}-k_{n-\ell+j}-k_{n+i-j}\\ &=n + (k_i-k_j)+(k_j-k_i+nD(i,j)-n)=n D(i,j),
	\end{align*} where the first equality follows from~\eqref{dij_alt_def} while the second one is obtained also using~\eqref{dij_alt_def} together with the conditions $D(j,\ell)=D(i,\ell)=1$.

   For the fifth point, combining $k_i-k_{i+1}+r=0$, $n+k_i-k_j-k_{i+n-j}=0$ and $k_{i+1}-k_j-k_{i+n+1-j}=0$ yields $k_{i+n-j}-k_{i+n+1-j}=n-r$, hence $D(n-j+i)=1$. 
   
   For the sixth point, we have \begin{align*} nD(i+1,j)&=n+k_{i+1}-k_j-k_{n+i+1-j}=n+k_{i+1}+r-k_{j+1}-k_{n+i+1-j}\\ &=r+k_{n+i-j}-(r+k_{n+i-j})+n+k_{i+1}+r-k_{j+1}-k_{n+i+1-j}\\ &=nD(n-j+i)+nD(i+1,j+1).\end{align*} But by the first point we have $D(i+1,j+1)=D(j)=0$, yielding the expected result. 
\end{proof}

\section{Cancellativity}\label{sec_cancellativity}

The aim of this section is to show that $\mathcal{M}(n,m)$ is both left- and right-cancellative. 

\subsection{Left-cancellativity}

\begin{lemma}\label{cube_left}
The monoid presentation $\langle \mathcal{S}~\vert~\mathcal{R}'\rangle$ for $\mathcal{M}(n,m)$ given in Corollary~\ref{cor_pres_left_c} satisfies the sharp $\theta$-cube condition for every triple $(\rho_i, \rho_j, \rho_k)$ of pairwise distinct elements of $\mathcal{S}$. 
\end{lemma}

\begin{proof}
For $i,j$ and $\ell$ pairwise distinct, we need to show that either both $\theta( \theta(\rho_{i},\rho_{j}), \theta(\rho_{i},\rho_{\ell}))$ and $\theta( \theta(\rho_{j},\rho_{i}), \theta(\rho_{j},\rho_{\ell}))$ are defined and equal as words in $\mathcal{S}^*$, or neither is defined. It is sufficient to distinguish three cases: the case~$i<j<\ell$, the case~$i<\ell<j$, and the case~$\ell<j<i$. The three remaining cases are indeed obtained for free by swapping the roles of $i$ and $j$. Recall the value of the syntactic right-complement $\theta$ given in~\eqref{syntactic_right}. 

\begin{itemize}
\item\noindent{\bf Case $i<j<\ell$}. We have 
\begin{align*} \theta (\theta(\rho_{i}, \rho_{j}), \theta(\rho_{i}, \rho_{\ell}))&=\theta\left(\rho_n^{qi+B(1,n-i+1)} \rho_{{j-i}}, \rho_n^{qi+B(1,n-i+1)} \rho_{\ell-i}\right)=\theta(\rho_{{j-i}}, \rho_{{\ell-i}})\\&=\rho_n^{q(j-i)+B(1,n-j+i+1)} \rho_{\ell-j}, \end{align*}
where for the middle equality we used the fact that for all $a,b,c\in\mathcal{S}^*$, we have $\theta(ab,ac)=\theta(b,c)$ (which is an easy consequence of the relations~\eqref{c}-\eqref{d}). We also have   
\begin{align*} \theta (\theta(\rho_{j}, \rho_{i}), \theta(\rho_{j}, \rho_{\ell}))&=\theta\left(\rho_n^{qi+B(j-i,n-i)}, \rho_n^{qj+B(1,n-j+1)} \rho_{\ell-j}\right).
	\end{align*}
But we have $B(1,n-j+1)=B(1,n-j+i+1)+B(j-i, n-i)$ by~\eqref{sumofbij} and $B(1,n-j+i+1) \geq 0,$ yielding \begin{align*}\theta (\theta(\rho_{j}, \rho_{i}), \theta(\rho_{j}, \rho_{\ell}))&=\theta\left(1,\rho_n^{q(j-i)+B(1,n-j+i+1)} \rho_{\ell-j}\right)=\rho_n^{q(j-i)+B(1,n-j+i+1)} \rho_{\ell-j}.\end{align*}
\item\noindent{\bf Case $i<\ell<j$}. We have
\begin{align*} \theta(\theta(\rho_{i}, \rho_{j}), \theta(\rho_{i}, \rho_{\ell}))&=\theta\left(\rho_n^{qi+B(1,n-i+1)} \rho_{{j-i}},\rho_n^{qi+B(1,n-i+1)} \rho_{\ell-i}\right)=\theta(\rho_{j-i}, \rho_{{\ell-i}})\\&=\rho_n^{q(\ell-i)+B(j-\ell, n-\ell+i)}, \end{align*}
and 
\begin{align*} \theta(\theta(\rho_{j}, \rho_{i}), \theta(\rho_{j}, \rho_{\ell}))&=\theta\left( \rho_n^{qi+B(j-i,n-i)}, \rho_n^{q\ell+B(j-\ell,n-\ell)}\right).\end{align*} But thanks to~\eqref{sumofbij} we have $0\leq B(j-\ell, n-\ell+i)=B(j-\ell, n-\ell)-B(j-i, n-i)$, yielding 
\begin{align*} \theta(\theta(\rho_{j}, \rho_{i}), \theta(\rho_{j}, \rho_{\ell}))&=\theta(1, \rho_n^{q(\ell-i)+B(j-\ell, n-\ell+i)})=\rho_n^{q(\ell-i)+B(j-\ell, n-\ell+i)}. \end{align*}
\item \noindent{\bf Case $\ell<j<i$}. We have 
\begin{align*} \theta(\theta(\rho_{i}, \rho_{j}), \theta(\rho_{i}, \rho_{\ell}))=\theta(\rho_n^{qj+B(i-j,n-j)}, \rho_n^{q\ell+B(i-\ell,n-\ell)}). \end{align*} But we have $B(i-\ell, n-\ell)\leq B(i-j,n-j)$, yielding \begin{align*}  \theta(\theta(\rho_{i}, \rho_{j}), \theta(\rho_{i}, \rho_{\ell}))=\theta(\rho_n^{q(j-\ell)+B(i-j,n-j)-B(i-\ell,n-\ell)},1)=1.\end{align*} 
We also have 
\begin{align*} \theta(\theta(\rho_{j}, \rho_{i}), \theta(\rho_{j}, \rho_{\ell}))=\theta(\rho_n^{qj+B(1,n-j+1)} \rho_{{i-j}}, \rho_n^{q\ell+B(j-\ell, n-\ell)}).\end{align*} But we have $B(j-\ell, n-\ell)\leq B(1, n-j+1)$, yielding \begin{align*} \theta(\theta(\rho_{j}, \rho_{i}), \theta(\rho_{j}, \rho_{\ell}))=\theta(\rho_n^{q(j-\ell)+B(1,n-j+1)-B(j-\ell, n-\ell)} \rho_{j-i},1)=1.\end{align*}
\end{itemize}
Hence in all cases we have $\theta( \theta(\rho_{i},\rho_{j}), \theta(\rho_{i},\rho_{\ell}))=\theta( \theta(\rho_{j},\rho_{i}), \theta(\rho_{j},\rho_{\ell}))$, which concludes the proof.\end{proof}

\begin{prop}[Left-cancellativity]\label{left_cancellative}
The monoid $\mathcal{M}(n,m)$ is left-cancellative. It admits conditional right-lcms. When it exists, the right-lcm of $a$ and $b\in \mathcal{M}(n,m)$ is equal to $a\theta(a,b)=b\theta(b,a)$. 
\end{prop}

\begin{proof}
The monoid $\mathcal{M}(n,m)$ is right-Noetherian (Lemma~\ref{noeth_new}) and the presentation $\langle \mathcal{S} \ \vert \ \mathcal{R}' \rangle$ satisfies the sharp $\theta$-cube condition for every triple of pairwise distinct elements of $\mathcal{S}$ (Lemma~\ref{cube_left}). We can therefore apply Proposition~\ref{cancellative_criterion} to conclude the proof.  
\end{proof}

\begin{cor}[Right-lcms of pairs of atoms]\label{cor_right_lcm}
	Let $1\leq i < j \leq n$. The right-lcm of $\rho_i$ and $\rho_j$ in $\mathcal{M}(n,m)$ is given by $$\rho_{i} \rho_n^{qi+B(1,n-i+1)} \rho_{{j-i}}=\rho_{j} \rho_n^{qi+B(j-i,n-i)}.$$
\end{cor}

\begin{proof} This follows immediately from the proposition above and the definition of $\theta(\rho_i, \rho_j)$ and $\theta(\rho_j, \rho_i)$ given in~\eqref{syntactic_right}. \end{proof}

\subsection{Right-cancellativity}

To show that $\mathcal{M}(n,m)$ is right-cancellative, we show that the opposite monoid $\mathcal{M}(n,m)^{\mathrm{op}}$ is left-cancellative, using the presentation $\langle \mathcal{T} \ \vert \ (\mathcal{R}'')^{\mathrm{op}}\rangle$ introduced in the paragraph after Corollary~\ref{cor_pres_right_c}. Before proving that this presentation satisfies the sharp $\eta$-cube condition, we prove some technical results on the value of the function $\eta$ on certain pairs of elements in the four following Lemmatas. We let $n,m,q,r$ be as before. Recall the value of the syntactic right-complement $\eta$ given in~\eqref{right_less} and~\eqref{right_great}. 

\begin{lemma}\label{lem_tech_left}
Let $1\leq i \leq j\leq n$. 
\begin{enumerate}
	\item $\eta(\tau_n^q \tau_i, \tau_j)=\left\{ \begin{array}{ll}
		(\tau_n^q \tau_1)^{n-j-1} & \mbox{if}~D(j)=0\text{~and~}D(i,j+1)=1,\\
		(\tau_n^q \tau_1)^{n-j} & \mbox{otherwise.}
	\end{array} 
\right.$
	\item Assume that $q\geq 2$ and $1\leq \ell < q$. Then $\eta(\tau_n^{q-\ell} \tau_i, \tau_j)=(\tau_n^q \tau_1)^{n-j}$.
		
\end{enumerate}  
\end{lemma}

\begin{proof}
	
	We show both statements simultaneously by decreasing induction on $j$ (note that if $q=1$ there is only the first statement, but the proof below still applies as point $2$ is not used in the induction in this case). If $j=n$ then $D(j)=1$ and $\eta(\tau_n^q \tau_i, \tau_n)=\eta(\tau_n^{q-1} \tau_i, 1)=1$, hence the formula in point $1$ holds true. Similarly $\eta(\tau_n^{q-\ell}\tau_i,\tau_n)=\eta(\tau_n^{q-\ell-1}\tau_i, 1)=1$. Hence the formula in point $2$ holds true. 
	
	Hence assume that $j<n$. We begin by showing that the formula in point $1$ holds true. We have\begin{align*} \eta(\tau_n^q\tau_i, \tau_j)&=\eta(\tau_i, \eta(\tau_n^q, \tau_j))=\eta(\tau_i, \eta(\tau_n^{q-1}, \eta(\tau_n, \tau_j)))=\eta(\tau_i, \eta(\tau_n^{q-1}, \tau_n^{q-1+D(j)}\tau_{j+1}))\\
	&=\eta(\tau_i, \eta(1, \tau_n^{D(j)}\tau_{j+1}))=\eta(\tau_i, \tau_n^{D(j)}\tau_{j+1}). \end{align*}
Note that $\eta(\tau_i, \tau_{j+1})=(\tau_n^q\tau_1)^{n-j-1}$ if $D(i,j+1)=1$ while $\eta(\tau_i, \tau_{j+1})=(\tau_n^q\tau_1)^{n-j}$ if $D(i,j+1)=0$. Hence it remains to check that when $D(j)=1$, we have $\eta(\tau_i, \tau_n \tau_{j+1})=(\tau_n^q\tau_1)^{n-j}$. In this case we have $$\eta(\tau_i, \tau_n \tau_{j+1})=\eta(\tau_i, \tau_n) \eta(\eta(\tau_n, \tau_i), \tau_{j+1})=
	\tau_n^q \tau_1 \eta(\tau_n^{q-1+D(i)}\tau_{i+1}, \tau_{j+1}). $$ Assume that $D(i)=0$. By induction we have $\eta(\tau_n^{q-1}\tau_{i+1}, \tau_{j+1})=(\tau_n^q\tau_1)^{n-j-1}$ if $q\neq 1$. If $q=1$ then by Lemma~\ref{ft}~(1) as $D(i)=0$ and $D(j)=1$ we have $D(i+1,j+1)=1$, hence by~\eqref{right_less} we also get $\eta(\tau_{i+1}, \tau_{j+1})=(\tau_n^q\tau_1)^{n-j-1}$. 
	Assume now that $D(i)=1$. Also by induction, we have $\eta(\tau_n^{q}\tau_{i+1}, \tau_{j+1})=(\tau_n^q\tau_1)^{n-j-1}$ except if $D(j+1)=0$ and $D(i+1, j+2)=1$. But since $D(i)=1$, by Lemma~\ref{ft}~(1) this situation cannot appear.

Let us now prove the formula in the second statement. We have \begin{align*} \eta(\tau_n^{q-\ell}\tau_i, \tau_j)&=\eta(\tau_i, \eta(\tau_n^{q-\ell}, \tau_j))=\eta(\tau_i, \eta(\tau_n^{q-\ell-1}, \eta(\tau_n, \tau_j)))= \eta(\tau_i, \eta(\tau_n^{q-\ell-1}, \tau_n^{q-1+D(j)}\tau_{j+1}))\\ &=\eta(\tau_i, \eta(1, \tau_n^{\ell+D(j)}\tau_{j+1}))=\eta(\tau_i,  \tau_n^{\ell+D(j)}\tau_{j+1})=\eta(\tau_i,\tau_n) \eta(\eta(\tau_n, \tau_i),  \tau_n^{\ell-1+D(j)}\tau_{j+1})\\
&=\tau_n^q\tau_1 \eta(\tau_n^{q-1+D(i)}\tau_{i+1}, \tau_n^{\ell-1+D(j)}\tau_{j+1})= \tau_n^q\tau_1 \eta(\tau_n^{q-\ell+D(i)-D(j)}\tau_{i+1}, \tau_{j+1}).
 \end{align*} If $D(i)=D(j)$, then by induction we get $\eta(\tau_n^{q-\ell}\tau_{i+1},\tau_{j+1})=(\tau_n^q\tau_1)^{n-j-1}$, which concludes the proof. If $D(i)=0$ and $D(j)=1$, then we also get the result by induction if $q-\ell-1\neq 0$; if $q-\ell-1=0$ we get it by~\eqref{right_less} as $D(i+1,j+1)=D(j)=1$ by Lemma~\ref{ft}~(1). If $D(i)=1$ and $D(j)=0$, then we again get the result by induction if $q-\ell+1\neq q$; if $q-\ell+1=q$, then we also get it by induction, since if $D(j+1)=0$, we have $D(i+1,j+2)=D(j+1)=0$ by Lemma~\ref{ft}~(1).\end{proof}

\begin{lemma}\label{lem_tech_right}
	Let $1\leq i < j\leq n$. 
	\begin{enumerate}
		\item $\eta(\tau_i, \tau_n^q \tau_j)=\left\{ \begin{array}{ll} \tau_n^q\tau_1 & \mbox{if~} i+1=j\mbox{~and~}D(i)=0,\\
			(\tau_n^q\tau_1)^{n-j+2} & \mbox{if~} i+1\neq j, D(i)=0\mbox{~and~}D(i+1,j)=0,\\
			(\tau_n^q\tau_1)^{n-j+1} & \mbox{otherwise.}
		 \end{array}\right.$
	 \item Assume that $q\geq 2$ and $1\leq \ell < q$. We have $\eta(\tau_i, \tau_n^{q-\ell} \tau_j)=(\tau_n^q\tau_1)^{n-j+1}.$
	 \end{enumerate}
\end{lemma}

	\begin{proof}
	As for Lemma~\ref{lem_tech_left}, we prove both statements simultaneously by decreasing induction on $j$. If $j=n$ then $\eta(\tau_i, \tau_n^{q+1})=\eta(\tau_i, \tau_n)\eta(\eta(\tau_n, \tau_i), \tau_n^q)$. But we have $\eta(\tau_i,\tau_n)=\tau_n^q\tau_1$ while $\eta(\tau_n, \tau_i)=\tau_n^{q-1+D(i)}\tau_{i+1}$. Hence if $D(i)=0$ we get $\eta(\tau_i, \tau_n^{q+1})=\tau_n^q\tau_1\eta(\tau_{i+1}, \tau_n)$ which yields $\tau_n^q\tau_1$ if $i+1=j=n$ and $(\tau_n^q\tau_1)^2$ otherwise. Note that $D(i+1,j)=D(i+1,n)=0$ in this last case. If $D(i)=1$ we get $\eta(\tau_i, \tau_n^{q+1})=\tau_n^q\tau_1\eta(\tau_{i+1}, 1)$ which yields $\tau_n^q \tau_1$ for all $i$. Hence the formula in the first point holds true for $j=n$. Now assume that $q\geq 2$ and $1\leq \ell < q$. We have \begin{align*}\eta(\tau_i, \tau_n^{q-\ell}\tau_n)&= \eta(\tau_i ,\tau_n)\eta(\eta(\tau_n, \tau_i), \tau_n^{q-\ell})=\tau_n^q\tau_1\eta(\tau_n^{q-1+D(i)}\tau_{i+1}, \tau_n^{q-\ell})=\tau_n^q \tau_1 \eta(\tau_n^{\ell-1+D(i)}\tau_{i+1},1)\\ &=\tau_n^q\tau_1,
	\end{align*} hence the formula in point $2$ also holds true for $j=n$. 
	
	Hence assume that $j<n$. We begin by proving the formula of the first point. We have \begin{align*}\eta(\tau_i,\tau_n^q\tau_j)&=\eta(\tau_i,\tau_n)\eta(\eta(\tau_n,\tau_i),\tau_n^{q-1}\tau_j) = \tau_n^q\tau_1 \eta(\tau_n^{q-1+D(i)}\tau_{i+1}, \tau_n^{q-1}\tau_j)=\tau_n^q\tau_1 \eta(\tau_n^{D(i)}\tau_{i+1},\tau_j).
		\end{align*} We first analyse the case $D(i)=0$. If $i+1=j$ then we get $\eta(\tau_i, \tau_n^q\tau_j)=\tau_n^q\tau_1$. Hence assume that $i+1\neq j$. By~\eqref{right_less} and the above computation we obtain $\eta(\tau_i,\tau_n^q\tau_j)=\tau_n^q\tau_1\eta(\tau_{i+1},\tau_j)=(\tau_n^q\tau_1)^{n-j+2}$ if $D(i+1,j)=0$ while $\eta(\tau_i,\tau_n^q\tau_j)=(\tau_n^q\tau_1)^{n-j+1}$ if $D(i+1,j)=1$, as expected. We now analyze the case $D(i)=1$. We calculate $$\eta(\tau_i, \tau_n^q\tau_j)=\tau_n^q\tau_1 \eta(\tau_n \tau_{i+1},\tau_j)=\tau_n^q\tau_1 \eta(\tau_{i+1}, \eta(\tau_n,\tau_j))=\tau_n^q\tau_1 \eta(\tau_{i+1},\tau_n^{q-1+D(j)}\tau_{j+1}).$$ If $D(j)=0$, then by induction we have using the second point that $\eta(\tau_{i+1}, \tau_n^{q-1} \tau_{j+1})=(\tau_n^q\tau_1)^{n-j}$ except possibly if $q=1$. If $q=1$, then as by Lemma~\ref{ft}~(1) we have $D(i+1,j+1)=0$, by~\eqref{right_less} we get $\eta(\tau_{i+1}, \tau_n^{q-1} \tau_{j+1})=(\tau_n^q\tau_1)^{n-j}$. 
 If $D(j)=1$, then by induction we have by the first point that $\eta(\tau_{i+1}, \tau_n^q\tau_{j+1})=(\tau_n^q\tau_1)^{n-j}$ except in two cases: if $(i+2=j+1, D(i+1)=0)$ or if $(i+2\neq j+1, D(i+1)=0, D(i+2,j+1)=0)$. The first case cannot happen since we get $j=i+1$ and $D(j)=1$ while $D(i+1)=0$, a contradiction. The second case also yields a contradiction: we have $D(j)=1$, $D(i+1)=0$, $D(i+2,j+1)=0$, which by Lemma~\ref{ft}~(1) is impossible. Hence the first formula holds true. 

We now prove the second one. We have \begin{align*} \eta(\tau_i, \tau_n^{q-\ell}\tau_j)&=\eta(\tau_i, \tau_n)\eta(\eta(\tau_n,\tau_i),\tau_n^{q-\ell-1}\tau_j)=\tau_n^q\tau_1 \eta(\tau_n^{q-1+D(i)}\tau_{i+1}, \tau_n^{q-\ell-1}\tau_j) = \tau_n^q\tau_1 \eta(\tau_n^{\ell+D(i)}\tau_{i+1}, \tau_j). 
\end{align*} If $D(i)=0$, by Lemma~\ref{lem_tech_left}~(2), since $1\leq \ell < q$, we have $\eta(\tau_n^{\ell}\tau_{i+1},\tau_j)=(\tau_n^q\tau_1)^{n-j}$. If $D(i)=1$, then if $\ell+1\neq q$, we get as well that $\eta(\tau_n^{\ell+1}\tau_{i+1},\tau_j)=(\tau_n^q\tau_1)^{n-j}$ by the same Lemma. In the case where $\ell+1=q$, we also have $\eta(\tau_n^{\ell+1}\tau_{i+1},\tau_j)=(\tau_n^q\tau_1)^{n-j}$ by Lemma~\ref{lem_tech_left} except possibly in the case where $D(j)=0$ and $D(i+1,j+1)=1$. But as $D(i)=1$, this case is impossible by Lemma~\ref{ft}~(1). Hence we get $\eta(\tau_i, \tau_n^{q-\ell}\tau_j)=(\tau_n^q\tau_1)^{n-j+1}$ in all cases. This concludes the proof.  
\end{proof}

\begin{lemma}\label{tech_tn_3}
Let $1\leq i < j \leq n$.
\begin{enumerate}
	\item 	$\eta(\tau_n^q\tau_j, \tau_i)=\left\{ \begin{array}{ll} (\tau_n^q\tau_1)^{n-j}\tau_n^{D(i+1,j)} \tau_{i+n-j+1} & \mbox{if~} D(i)=1,\\
		\eta(\tau_j, \tau_{i+1})  & \mbox{if~} D(i)=0.
	\end{array}\right.$
	\item Assume that $q\geq 2$ and $1\leq \ell < q$. We have $$\eta(\tau_n^{q-\ell} \tau_j, \tau_i)=(\tau_n^q\tau_1)^{n-j} \tau_n^{\ell-1+D(i+1,j)+D(i)} \tau_{i+n-j+1}.$$
\end{enumerate}

\end{lemma}

\begin{proof}

	We again argue by decreasing induction on $j$, proving both statements simultaneously. For $j=n$ we have $\eta(\tau_n^q \tau_n, \tau_i)=\eta(\tau_n^q, \eta(\tau_n, \tau_i))=\eta(\tau_n^q, \tau_n^{q-1+D(i)}\tau_{i+1}),$ yielding $\tau_{i+1}$ if $D(i)=1$ and $\eta(\tau_n, \tau_{i+1})$ if $D(i)=0$. Since we have $D(i+1,n)=0$, the first formula holds true for $j=n$. A similar calculation yields $\eta(\tau_n^{q-\ell}\tau_n, \tau_i)=\tau_n^{\ell-1+D(i)}\tau_{i+1}$. Since $D(i+1,n)=0$ we also get that the second formula holds true for $j=n$. 

Now assume that $j<n$. We begin by showing the formula of the first point. We have 
\begin{align*}
	\eta(\tau_n^q\tau_j, \tau_i)=\eta(\tau_n^{q-1}\tau_j, \eta(\tau_n,\tau_i))=\eta(\tau_n^{q-1}\tau_j,\tau_n^{q-1+D(i)}\tau_{i+1})=\eta(\tau_j,\tau_n^{D(i)}\tau_{i+1})
\end{align*} 
We obtain the claimed formula if $D(i)=0$. Hence assume that $D(i)=1$. In this case we have \begin{align}\label{inter_1}
	\eta(\tau_n^q\tau_j, \tau_i)=\eta(\tau_j, \tau_n \tau_{i+1})=\eta(\tau_j, \tau_n) \eta(\eta(\tau_n, \tau_j), \tau_{i+1})=\tau_n^q \tau_1 \eta(\tau_n^{q-1+D(j)}\tau_{j+1},\tau_{i+1})
\end{align}
If $D(j)=0$ then by induction we have $\eta(\tau_n^{q-1}\tau_{j+1},\tau_{i+1})=(\tau_n^q\tau_1)^{n-j-1} \tau_n^{D(i+2,j+1)+D(i+1)} \tau_{i+n-j+1}$ if $q\neq 1$.  Now by Lemma~\ref{ft}~(3) we have \begin{eqnarray}\label{dij} D(i+2, j+1)+D(i+1)=D(i+1,j)+D(j),
\end{eqnarray}
yielding the expected formula in this case. If $q=1$ then we have $\eta(\tau_n^q\tau_j, \tau_i))=(\tau_n^q\tau_1)\eta(\tau_{j+1}, \tau_{i+1})=(\tau_n^q\tau_1)(\tau_n^q\tau_1)^{n-j-1} \tau_n^{D(n-j+i)}\tau_{n-j+i+1}$. But by Lemma~\ref{ft}~(6), we have $D(n-j+i)=D(i+1,j)$, again yielding the expected formula. 

If $D(j)=1$ then by induction we get that $\eta(\tau_n^{q}\tau_{j+1},\tau_{i+1})=(\tau_n^q\tau_1)^{n-j-1}\tau_n^{D(i+2,j+1)} \tau_{i+n-j+1} $ if $D(i+1)=1$. This is also equal to $(\tau_n^q\tau_1)^{n-j-1}\tau_n^{D(i+1,j)} \tau_{i+n-j+1} $ by~\eqref{dij} since $D(i+1)=D(j)=1$, hence from~\eqref{inter_1} we get the claimed formula. In the remaining case, we have $D(i)=1=D(j)$ and $D(i+1)=0$. Hence by Lemma~\ref{ft}~(1) we get $D(i+2,j+1)=1$ (hence in particular $i+2 < j+1$). This yields by induction and~\eqref{right_great} that $\eta(\tau_n^q \tau_{j+1}, \tau_{i+1})=\eta(\tau_{j+1},\tau_{i+2})=(\tau_n^q\tau_1)^{n-j-1} \tau_{n-j+i+1}$ which again concludes the proof with~\eqref{inter_1} as~\eqref{dij} yields $D(i+1,j)=0$.

We now show the formula in the second point. We have      
\begin{align*}
	\eta(\tau_n^{q-\ell}\tau_j, \tau_i)&=\eta(\tau_n^{q-\ell-1}\tau_j, \eta(\tau_n,\tau_i))=\eta(\tau_n^{q-\ell-1}\tau_j,\tau_n^{q-1+D(i)}\tau_{i+1})=\eta(\tau_j,\tau_n^{\ell+D(i)}\tau_{i+1})\\ &=\eta(\tau_j,\tau_n)\eta(\eta(\tau_n,\tau_j),\tau_n^{\ell-1+D(i)}\tau_{i+1})=\tau_n^q \tau_1\eta(\tau_n^{q-1+D(j)},\tau_n^{\ell-1+D(i)}\tau_{i+1})\\ &=\tau_n^q\tau_1\eta(\tau_n^{q-\ell+D(j)-D(i)}\tau_{j+1},\tau_{i+1})
\end{align*}
Except possibly for two particular cases which we will check afterwards, namely the cases $(\ell=q-1, D(j)=0, D(i)=1)$ and $(\ell=1, D(j)=1, D(i)=0$), by induction we get
\begin{align*}
	\eta(\tau_n^{q-\ell}\tau_j, \tau_i) &=\tau_n^q\tau_1\eta(\tau_n^{q-\ell+D(j)-D(i)}\tau_{j+1},\tau_{i+1})=
		(\tau_n^q\tau_1)^{n-j}\tau_n^{\ell-1+D(i)-D(j)+D(i+2,j+1)+D(i+1)}\tau_{i+n-j+1} \\ &=(\tau_n^q\tau_1)^{n-j}\tau_n^{\ell-1+D(i+1,j)+D(i)}\tau_{i+n-j+1},
\end{align*}
where the last equality follows from~\eqref{dij}. We thus get the claim in this case. We now check the aforementioned two remaining cases. 

If $D(i)=1$, $D(j)=0$ and $\ell=q-1$ we have $D(i+1,j+1)=0$ by Lemma~\ref{ft}~(1) and hence using~\eqref{right_great} we get \begin{align*}\eta(\tau_n^{q-\ell}\tau_j, \tau_i)&=\tau_n^q\tau_1 \eta(\tau_{j+1},\tau_{i+1})=(\tau_n^q \tau_1)^{n-j} \tau_n^{q-1+D(n-j+i)}\tau_{n-j+i+1}\\ &=(\tau_n^q \tau_1)^{n-j} \tau_n^{\ell-1+D(n-j+i)+D(i)}\tau_{n-j+i+1}.\end{align*} It remains to check that $D(n-j+i)=D(i+1,j)$, which we get by Lemma~\ref{ft}~(6). 

If $D(i)=0, D(j)=1$ and $\ell=1$, by induction we have $$\eta(\tau_n^{q-\ell}\tau_j, \tau_i)=\tau_n^q\tau_1 \eta(\tau_n^q\tau_{j+1},\tau_{i+1})=\left\{ \begin{array}{ll}  \tau_n^q\tau_1\eta(\tau_{j+1},\tau_{i+2}) & \mbox{if~} D(i+1)=0,\\
	(\tau_n^q \tau_1)^{n-j}\tau_n^{D(i+2,j+1)}\tau_{i+n-j+1} & \mbox{if~} D(i+1)=1.
\end{array}\right.$$ In the first situation by Lemma~\ref{ft}~(1) we get $D(i+2,j+1)=1$ (in particular $i+2<j+1$). We thus get $\eta(\tau_{j+1},\tau_{i+2})=(\tau_n^q\tau_1)^{n-j-1} \tau_{n-j+i+1}$ which concludes the proof since by~\eqref{dij} we have $$D(i+1,j)+D(i)=D(i+2,j+1)+D(i+1)-D(j)+D(i)=0,$$ while in the second situation the above equality yields $D(i+1,j)+D(i)=D(i+2,j+1)$ which also concludes the proof.
\end{proof}

\begin{lemma}\label{last_tech}
	Let $1\leq i \leq j \leq n$. Then we have
	\begin{enumerate}
		\item $\eta(\tau_j, \tau_n^q \tau_i)=\left\{ \begin{array}{ll} \tau_n^q\tau_1\eta(\tau_{j+1}, \tau_i) & \mbox{if~} D(j)=0,\\
			(\tau_n^q\tau_1)^{n-j} \tau_n^{q-1+D(i+1,j+1)+D(i)-D(j)}\tau_{i+n-j} & \mbox{if~} D(j)=1.
		\end{array}\right.$
	\item  Assume that $q\geq 2$ and $1\leq \ell < q$. We have $$\eta(\tau_j, \tau_n^{q-\ell}\tau_i)=(\tau_n^q\tau_1)^{n-j} \tau_n^{q-\ell-D(j)-1+D(i+1,j+1)+D(i)} \tau_{i+n-j}.$$
	\end{enumerate}
\end{lemma}

\begin{proof}
We again argue by decreasing induction on $j$. We begin by proving both points for $j=n$. For the first point, if $j=n$ then $D(j)=1$ and $\eta(\tau_n, \tau_n^q \tau_i)=\tau_n^{q-1}\tau_i$, hence the claimed formula holds true since an immediate calculation shows that $D(i+1,n+1)+D(i)=1$ (recall that $k_{n+1}=k_1=r$). For the second point, if $j=n$ we have $\eta(\tau_n, \tau_n^{q-\ell} \tau_i)=\tau_n^{q-\ell-1} \tau_i$ while as above $D(i+1,n+1)+D(i)-D(j)=0$.

Hence assume that $j\neq n$. We prove the first point. We have \begin{align*}
	\eta(\tau_j, \tau_n^q\tau_i)&=\eta(\tau_j, \tau_n)\eta(\eta(\tau_n,\tau_j), \tau_n^{q-1} \tau_i)=\tau_n^q\tau_1 \eta(\tau_n^{q-1+D(j)}\tau_{j+1}, \tau_n^{q-1}\tau_i)=\tau_n^q\tau_1 \eta(\tau_n^{D(j)}\tau_{j+1}, \tau_i).
\end{align*}
If $D(j)=0$ we get the claimed formula. If $D(j)=1$ we also get the claimed formula if $q\neq 1$ by applying Lemma~\ref{tech_tn_3}~(2) with $\ell=q-1$. If $q=1$ we use Lemma~\ref{tech_tn_3}~(1), which yields \begin{align*}
	\tau_n^q\tau_1 \eta(\tau_n\tau_{j+1}, \tau_i)=\left\{ \begin{array}{ll} \tau_n^q\tau_1(\tau_n^q\tau_1)^{n-j-1} \tau_n^{D(i+1,j+1)} \tau_{i+n-j} & \mbox{if~} D(i)=1,\\
		\tau_n^q\tau_1 \eta(\tau_{i+1}, \tau_{j+1}) & \mbox{if~} D(i)=0.
	\end{array}\right.
\end{align*}
In the first case we get the claimed formula. In the second case, since $D(j)=1$ and $D(i)=0$ (hence $i\neq j$), by Lemma~\ref{ft}~(1) we have $D(i+1,j+1)=1$, hence~\eqref{right_great} yields $\eta(\tau_{i+1},\tau_{j+1})=(\tau_n^q\tau_1)^{n-j-1} \tau_{n-j+i}$, from what we also obtain the expected formula. 

For the second point we have 
\begin{align*}
	\eta(\tau_j, \tau_n^{q-\ell}\tau_i)&=\eta(\tau_j, \tau_n)\eta(\eta(\tau_n,\tau_j), \tau_n^{q-\ell-1} \tau_i)=\tau_n^q\tau_1 \eta(\tau_n^{q-1+D(j)}\tau_{j+1}, \tau_n^{q-\ell-1}\tau_i)=\tau_n^q\tau_1 \eta(\tau_n^{D(j)+\ell} \tau_{j+1}, \tau_i)\\ &=\left\{ \begin{array}{ll}
		\tau_n^q\tau_1 \eta(\tau_{j+1}, \tau_{i+1}) & \mbox{if~} \ell=q-1, D(j)=1,D(i)=0\\  
		(\tau_n^q\tau_1)^{n-j} \tau_n^{q-\ell-D(j)-1+D(i+1,j+1)+D(i)} \tau_{i+n-j} & \mbox{otherwise}.\\
	\end{array}\right.,
\end{align*}
where the last equality is obtained by applying Lemma~\ref{tech_tn_3} again. In the first case, as $D(j)=1$ and $D(i)=0$ (hence $i\neq j$) by Lemma~\ref{ft}~(1) we get that $D(i+1,j+1)=1$, yielding by~\ref{right_great} $$\tau_n^q \tau_1 \eta(\tau_{j+1}, \tau_{i+1})=(\tau_n^q \tau_1)^{n-j} \tau_{n-j+i},$$ and $q-\ell-D(j)-1+D(i+1,j+1)+D(i)=0$ in this case, hence we get  the claimed formula in all cases. 
\end{proof}

The verification of the sharp $\eta$-cube condition is technical and involves distinguishing a lot of cases using the above four Lemmatas; we therefore threat it in Appendix~\ref{sec_appendix}. 

\begin{prop}[Right-cancellativity]\label{right_cancellative}
	The monoid $\mathcal{M}(n,m)^{\mathrm{op}}$ is left-cancellative and admits conditional right-lcms. Equivalently, the monoid $\mathcal{M}(n,m)$ is right-cancellative and admits conditional left-lcms.  
\end{prop}
\begin{proof}
	Since $\mathcal{M}(n,m)^{\mathrm{op}}$ is right-Noetherian (because $\mathcal{M}(n,m)$ is left-Noetherian by Lemma~\ref{noeth_new}) and the presentation $\langle\mathcal{T}, {(\R'')}^{\mathrm{op}}\rangle$ satisfies the sharp $\eta$-cube condition for every triple of pairwise distinct elements of $\mathcal{T}$ (Lemmatas~\ref{right_cube_1}, \ref{right_cube_2} and \ref{right_cube_3}), Proposition~\ref{cancellative_criterion} ensures that $\mathcal{M}(n,m)^{\mathrm{op}}$ is left-cancellative and admits conditional right-lcms. 
\end{proof}

\begin{cor}[Left-lcms of pairs of atoms]\label{cor_common_right}
Let $1\leq i < j \leq n$. The left-lcm of $\rho_i$ and $\rho_j$ is given by \begin{align*} 
\left\{ \begin{array}{ll}
	(\rho_1\rho_n^q)^{n-j} \rho_i=\rho_{n-j+i} (\rho_1\rho_n)^{n-j} \rho_j & \mbox{if~} D(i,j)=1,\\  
	(\rho_1\rho_n^q)^{n-j+1} \rho_i=\rho_{n-j+i+1} (\rho_1\rho_n)^{n-j} \rho_n^{q-1+D(n-j+i)} \rho_j & \mbox{if~} D(i,j)=0.\\
\end{array}\right.
\end{align*}
\end{cor}

\begin{proof}
By Proposition~\ref{cancellative_criterion}, the right-lcm of $a$ and $b$ in $\mathcal{M}(n,m)^{\mathrm{op}}$ exists if and only if $\eta(a,b)$ is defined, and is then given by $a \eta(a,b)=b\eta(b,a)$. For $\tau_i, \tau_j$ we know that $\eta(\tau_i, \tau_j)$ is defined (see~\eqref{right_less} and~\eqref{right_great}), hence the right-lcm of $\tau_i$ and $\tau_j$ is given by $\tau_i \eta(\tau_i, \tau_j)=\tau_j\eta(\tau_j, \tau_i)$. It then suffices to reverse the obtained words and replace $\tau_k$'s by $\rho_k$'s to obtain the left-lcm of $\rho_i$ and $\rho_j$ in $\mathcal{M}(n,m)$.  
\end{proof}

\subsection{Least common multiple of the atoms}\label{section_lcm_atoms}

We now determine the left- and righ-lcm of the atoms in $\mathcal{M}(n,m)$. An interesting feature of $\mathcal{M}(n,m)$ is that for some values of $n$ and $m$, the left- and right-lcm of the atoms are distinct:

\begin{prop}[Least common multiple of the atoms in $\mathcal{M}(n,m)$]
We have\begin{enumerate} 
\item The right-lcm of the generators $\rho_i$, $1\leq i \leq n$ of $\mathcal{M}(n,m)$ is given by $\rho_n^{m-q}$.
\item The left-lcm of the generators $\rho_i$, $1\leq i \leq n$ of $\mathcal{M}(n,m)$ is given by \begin{align*}\left\{ \begin{array}{ll}
	\rho_{n-1} \rho_n^{q(n-2)+r-1} & \mbox{if~} D(1)=1,\\  
	\rho_n^{m-q} & \mbox{if~} D(1)=0.\\
\end{array}\right.\end{align*}
\end{enumerate}
\end{prop}

\begin{proof}
We first prove that $\rho_n^{m-q}$ is the right-lcm of the atoms. By Corollary~\ref{cor_right_lcm}, the right-lcm of $\rho_n$ and $\rho_{n-1}$ is given by $\rho_n \rho_n^{q(n-1)+B(1,1)}$. But $B(1,1)=r-1$, hence $\rho_n \rho_n^{q(n-1)+B(1,1)}=\rho_n^{m-q}$. Hence to conclude the proof, it suffices to show that $\rho_i$ left-divides $\rho_n^{m-q}$ for all $1\leq i \leq n-2$. This is the case, as $\rho_n^{m-q}=(\rho_1 \rho_n^q)^{n-1}\rho_1$ (Lemma~\ref{lem_red_garside}), while $(\rho_1 \rho_n^q)^i=\rho_i \rho_n^{qi+B(1,n-i+1)}$ (Lemma~\ref{lem_rel_left}~(2)). 

We now determine the left-lcm of the atoms. First assume that $D(1)=0$. By Corollary~\ref{cor_common_right}, the left-lcm of $\rho_1$ and $\rho_2$ is given by $(\rho_1 \rho_n^q)^{n-1} \rho_1$ since $D(1,2)=D(1)=0$. But by Lemma~\ref{lem_red_garside} we have $(\rho_1 \rho_n^q)^{n-1} \rho_1=\rho_n^{m-q}$. To conclude the proof in this case, it suffices to show that $\rho_j$ is a right-divisor of $\rho_n^{m-q}$ for all $2 < j \leq n-1$. But by Lemma~\ref{lem_rel_left}~(1), for $2< j \leq n-1$ we have $$(\rho_1 \rho_n^q)^{n-j} \rho_j=\rho_n \rho_n^{q(n-j)+B(j,j)}=\rho_n^{q(n-j)+1+B(j,j)},$$ which right-divides $\rho_n^{m-q}$ as $1+B(j,j)\leq r$ and $j\geq 2$. Hence $\rho_j$ is a right-divisor of $\rho_n^{m-q}$ also for $2< j \leq n-1$. 

Now assume that $D(1)=1$. By Corollary~\ref{cor_common_right}, the left-lcm of $\rho_1$ and $\rho_2$ is given by $(\rho_1 \rho_n^q)^{n-2} \rho_1$. By Lemma~\ref{lem_rel_left}~(1), this is equal to $\rho_{n-1}\rho_n^{q(n-2)+B(1,2)}$, and $B(1,2)$ is the number of bad indices in $\{1, 2, \cdots, n-2\}$, which is $r-1$ since $n-1$ is good but $n$ is bad. We then conclude the proof as in the case $D(1)=0$, noting that for $2<j\leq n-1$, we have $q(n-j)+1+B(j,j)\leq q(n-j)+r= q(n-j)+1+r-1 \leq q(n-2)+r-1$. 
\end{proof}

\begin{exple}
Consider $n=3$ and $m=5$ as in Example~\ref{pres_3_5}. Then $q=1$ and $r=2$, and $1$ is bad. It follows from the above proposition that the right-lcm of $\rho_1, \rho_2$ and $\rho_3$ is equal to $\rho_3^{4}$, while their left-lcm is given by $\rho_2 \rho_3^{2}$. We indeed see from the presentation given in Example~\ref{pres_3_5} that $\rho_1, \rho_2$ and $\rho_3$ all right-divide $\rho_2 \rho_3^2$ as $$\rho_1 \rho_3 \rho_1=\rho_2 \rho_3^2=\rho_2 \rho_1 \rho_3 \rho_2.$$
\end{exple}

\section{Garside structure}\label{sec_garside_structure}

\begin{nota}
Let $\mathcal{M}(n,m)$ with its presentation given in~\eqref{pres_pratique}. We set $\Delta:=\rho_n^{m}$, omitting the dependency on $n$ and $m$. 
\end{nota}

\begin{lemma}\label{lem_garside_elt}
Let $1\leq i \leq n$. Let $a_i:=\rho_n^{qi+B(1,n-i+1)} (\rho_1 \rho_n^q)^{n-i}$. Then $a_i \rho_i=\rho_i a_i=\Delta$. In particular, every element of $\mathcal{S}=\{\rho_1, \rho_2, \dots, \rho_n\}$ is both a left- and a right-divisor of $\Delta$, and $\Delta$ is central in $\mathcal{M}(n,m)$.  
\end{lemma}

\begin{proof}
Using Lemmatas~\ref{lem_red_garside} and~\ref{lem_rel_left}~(2) we have $$\Delta=(\rho_1\rho_n^q)^n=(\rho_1\rho_n^q)^i (\rho_1\rho_n^q)^{n-i}=\rho_{i} \rho_n^{qi+B(1,n-i+1)}(\rho_1\rho_n^q)^{n-i}=\rho_i a_i.$$ On the other hand, applying Lemma~\ref{lem_rel_left}~(2) we have \begin{align*} a_i \rho_i&=\rho_n^{qi+B(1,n-i+1)} (\rho_1 \rho_n^q)^{n-i} \rho_i=\rho_n^{qi+B(1,n-i+1)} \rho_n \rho_n^{q(n-i)+B(i,i)}. 
\end{align*} Observing that $B(1,n-i+1)$ counts the number of bad indices in $\{1, 2, \dots, i-1\}$, that $B(i,i)$ counts the number of bad indices in $\{i, i+1, \dots, n-1\}$, and that $n$ is bad, we get that $1+B(1,n-i+1)+B(i,i)$ is equal to the total number of bad indices in $\{1, 2, \dots, n\}$, which is equal to $r$. Hence $a_i\rho_i=\rho_n^{qn+r}=\rho_n^m=\Delta$. 

Since $\mathcal{S}$ generates $\mathcal{M}(n,m)$, the claim that $\Delta$ is central follows by cancellativity.  
\end{proof}

\begin{cor}[Garside element in $\mathcal{M}(n,m)$]\label{garside_elt}
The left- and right-divisors of $\Delta$ in $\mathcal{M}(n,m)$ coincide, and form a finite set, which we denote $\mathrm{Div}(\Delta)$.
\end{cor}

\begin{proof}
Let $a, b\in\mathcal{M}(n,m)$ such that $ab=\Delta$. By the previous Lemma $\Delta$ is central in $\mathcal{M}(n,m)$. We hence have $$a\Delta=\Delta a= a ba.$$ By cancellativity (see Section~\ref{sec_cancellativity}) we get that $\Delta=ba$, hence left- and right-divisors of $\Delta$ coincide. The fact that $\mathrm{Div}(\Delta)$ is finite follows immediately from Lemma~\ref{noeth_new}.
\end{proof}

We can now prove the main result of the paper:

\begin{theorem}[New Garside structure on torus knot groups]\label{main_text}
The pair $(\mathcal{M}(n,m), \Delta)$ is a Garside monoid with Garside group isomorphic to the torus knot group $G(n,m)$. 
\end{theorem}	

\begin{proof}
The monoid $\mathcal{M}(n,m)$ is cancellative and admits conditional lcm's by Propositions~\ref{left_cancellative} and~\ref{right_cancellative}. It has Noetherian divisibility by Lemma~\ref{noeth_new}. By Corollary~\ref{garside_elt} and since, by Lemma~\ref{lem_garside_elt}, divisors of $\Delta$ include the generating set $\mathcal{S}$, the element $\Delta$ satisfies the last two conditions of Definition~\ref{def_garside}. We then get the existence of lcm's from the existence of conditional lcm's, applying Lemma~\ref{cond_delta_lcm}.

By Theorem~\ref{thm:ore}, we get that the Garside group $G(\mathcal{M}(n,m))$ has the same presentation as $\mathcal{M}(n,m)$, hence we conclude the proof using Proposition~\ref{isom_torus}.
\end{proof}

\section{Groups with analogous Garside structures}\label{sec_groups_analogous}

\subsection{The complex braid group of $G_{13}$}

As mentioned in Example~\ref{exple_g12}, the complex braid group of $G_{12}$ is isomorphic to $G(3,4)$ (see~\cite{Bannai}, or~\cite{BMR}) and hence admits the presentation $$B_{G_{12}}\cong ~\bigg\langle x_1, x_2, x_3 \ \bigg\vert\ 
\begin{matrix}
	x_1 x_2 x_3 x_1\\ = x_2 x_3 x_1 x_2\\
	=x_3 x_1 x_2 x_3 \end{matrix} \ \bigg\rangle.$$
In terms of complex braid groups, the generators $x_i$'s are so-called braided reflections. In terms of the presentation with the $\rho_i$'s one has $\rho_1\mapsto x_1$, $\rho_2 \mapsto x_3 x_1$, $\rho_3\mapsto x_2 x_3 x_1$. 

The complex braid group $B_{G_{13}}$ of $G_{13}$ admits a similar presentation with generators also given by braided reflections (see~\cite{Bannai, BMR}): $$B_{G_{13}}\cong\bigg\langle x_1, x_2, x_3 \ \bigg\vert\ 
\begin{matrix}
	x_1 x_2 x_3 x_1 x_2\\ = x_2 x_3 x_1 x_2 x_3,\\
	x_3 x_1 x_2 x_3\\ =x_1 x_2 x_3 x_1
\end{matrix}
\ \bigg\rangle.$$
As noticed by Picantin~\cite[Exemple 13]{Picantin}, this presentation is Garside. A similar Garside structure to the one constructed in this paper for torus knot groups can be constructed for $G_{13}$. Namely, the assignment $\rho_1 \mapsto x_1$, $\rho_2\mapsto x_2 x_1$, $\rho_3\mapsto x_2 x_3 x_1$ yields the presentation \begin{align}\label{pres_g13}B_{G_{13}}\cong\bigg\langle \rho_1, \rho_2, \rho_3 \ \bigg\vert\ 
\begin{matrix}
	\rho_2 \rho_3 \rho_1 = \rho_3^2\\
	\rho_1 \rho_3 \rho_2=\rho_3^2
\end{matrix} \ \bigg\rangle.\end{align}
We shall check in Proposition~\ref{dih_garside_thm} below that this presentation is Garside, as part of a bigger family of Garside presentations. To be more precise, the complex braid group of $G_{13}$ is isomorphic to the Artin group of dihedral type $I_2(6)$ (see~\cite{Bannai}). The standard presentation of $B_{I_2(6)}$ is given by \begin{align}\label{std_g2} B_{I_2(6)}\cong \langle \ \sigma, \tau \ \vert \ \sigma\tau\sigma\tau\sigma\tau=\tau\sigma\tau\sigma\tau\sigma \ \rangle
\end{align} and one can check directly that an isomorphism is given by $\sigma \mapsto (x_1 x_2 x_3 x_1)^{-1}$, $\tau\mapsto x_1$ (with inverse $x_1 \mapsto \tau$, $x_2\mapsto (\sigma \tau \sigma \tau \sigma \tau)^{-1} \sigma^2$, $x_3\mapsto \sigma^{-1} \tau \sigma$). Hence one passes from Presentation~\ref{pres_g13} to Presentation~\ref{std_g2} by $\rho_1 \mapsto \tau$, $\rho_2 \mapsto \sigma\tau^{-1} \sigma^{-1}\tau^{-1}\sigma^{-1}$, $\rho_3 \mapsto \tau^{-1} \sigma^{-1}$. Somewhat surprisingly, one can generalize Presentation~\ref{pres_g13} to a Garside presentation for all dihedral Artin groups of even type. We do it in the following subsection.  

\subsection{Dihedral Artin groups of even type}

Let $n\geq 1$. Consider the monoid presentation
\begin{equation}\label{dih_even}
\bigg\langle \tau_1, \tau_2, \rho \ \bigg\vert\ 
\begin{matrix}
\tau_1 \rho \tau_2=\rho^2\\
\tau_2 \rho^{n} \tau_1=\rho^{n+1}
\end{matrix}
\ \bigg\rangle
\end{equation}

Note that for $n=1$, this is the same as Presentation~\ref{pres_g13} from the previous section. We add the relation $\tau_2 \rho^n \tau_1=\tau_1 \rho \tau_2 \rho^{n-1}$, which is a consequence of the two defining relations, to the above presentation to get the presentation: \begin{equation}\label{dih_even_left_full}
\bigg\langle \tau_1, \tau_2, \rho \ \bigg\vert\ 
\begin{matrix}
\tau_1 \rho \tau_2=\rho^2\\
\tau_2 \rho^{n} \tau_1=\rho^{n+1}\\
\tau_2 \rho^n \tau_1=\tau_1 \rho \tau_2 \rho^{n-1}
\end{matrix}
\ \bigg\rangle
\end{equation}

\begin{rmq}
Presentation~\ref{dih_even} still makes sense for $n=0$, but in that case $\rho$ is not an atom of the corresponding monoid as $\rho=\tau_2 \tau_1$. In this case the obtained monoid is nothing but the Artin monoid of type $B_2=I_2(4)$. The monoid defined by Presentations~\ref{dih_even} or \ref{dih_even_left_full} will be In Proposition~\ref{dih_garside_thm} below to be a Garside monoid with corresponding Garside group isomorphic to the Artin group of type $I_2(2n+4)$, but we distinguish the case $n>1$ from the case $n=0$ which is not new and where the number of atoms differs.
\end{rmq}
Recall that the Artin group of dihedral type $I_2(4+2n)$ ($n\geq 0$) has standard presentation \begin{align}\label{dih_even_artin} B_{I_2(4+2n)}\cong \langle \ \sigma, \tau \ \vert \ (\sigma\tau)^{n+2}=(\tau\sigma)^{n+2} \ \rangle\end{align}

It is straighforward to check the following:

\begin{lemma}\label{isom_dih_even}
The group defined by Presentation~\ref{dih_even} is isomorphic to the Artin group of type $I_2(4+2n)$ via $\tau_1 \mapsto \tau$, $\tau_2 \mapsto \sigma\tau^{-1} \sigma^{-1} \tau^{-1} \sigma^{-1}$, $\rho \mapsto \tau^{-1} \sigma^{-1}$. The inverse map is given by $\tau \mapsto \tau_1$, $\sigma \mapsto (\tau_1 \rho)^{-1}$.
\end{lemma}

Presentation~\ref{dih_even_left_full} is right-complemented. The monoid defined by such a presentation is Noetherian: put $\lambda(\tau_i)=1$, $\lambda(\rho)=2$. Note that it is left-cancellative if and only if it is right-cancellative, as Presentation~\ref{dih_even} is symmetric up to reversing the roles of $\tau_1$ and $\tau_2$. 

The syntactic right-complement $\theta$ attached to Presentation~\ref{dih_even_left_full} is given by $$\theta(\tau_1, \rho)=\rho \tau_2, ~\theta(\rho, \tau_1)=\rho,~\theta(\tau_2, \rho)=\rho^n \tau_1,~\theta(\rho, \tau_2)=\rho^n,~\theta(\tau_1, \tau_2)=\rho\tau_2 \rho^{n-1},~\theta(\tau_2, \tau_1)=\rho^n \tau_1.$$

\begin{lemma}\label{dih_cube}
Presentation~\ref{dih_even_left_full} satisfies the sharp $\theta$-cube condition for every triple of pairwise distinct generators. 
\end{lemma}

\begin{proof}
We have $$\theta( \theta(\tau_1, \tau_2), \theta(\tau_1, \rho))=\theta(\rho\tau_2 \rho^{n-1}, \rho \tau_2)=\theta(\rho^{n-1},1)=1.$$
Now $$\theta( \theta(\tau_2, \tau_1), \theta(\tau_2, \rho))=\theta(\rho^n \tau_1, \rho^n \tau_1)=1.$$
Hence $\theta( \theta(\tau_1, \tau_2), \theta(\tau_1, \rho))=\theta( \theta(\tau_2, \tau_1), \theta(\tau_2, \rho)).$
Similarly one checks that $$\theta( \theta(\tau_1, \rho), \theta(\tau_1, \tau_2))=\rho^{n-1}=\theta( \theta(\rho, \tau_1), \theta(\rho, \tau_2)),$$
$$\theta( \theta(\tau_2, \rho), \theta(\tau_2, \tau_1))=1=\theta( \theta(\rho, \tau_2), \theta(\rho, \tau_1)),$$ which concludes the proof.
\end{proof}

Applying Proposition~\ref{cancellative_criterion} we get:

\begin{cor}\label{cor_dih_cancel}
The monoid defined by Presentation~\ref{dih_even} is left- and right-cancellative and admits conditional (left- and right-) lcms. 
\end{cor}

\begin{lemma}\label{dih_garside_elt}
Let $\Delta:=\rho^{n+2}$. Then $\Delta$ is a central Garside element in the monoid with Presentation~\ref{dih_even}.
\end{lemma}

\begin{proof}
It suffices to show that for every generator $x\in \{\tau_1, \tau_2, \rho\}$, there is ane element $y$ of the monoid such that $xy=\Delta=yx$. This is clear for $x=\rho$ as $\Delta$ is a power of $\rho$. For $x=\tau_1$ we have $\Delta=\tau_1 \rho \tau_2 \rho^{n}=xy$ with $y=\rho \tau_2 \rho^{n}$. We have $yx=\rho \tau_2 \rho^{n} \tau_1= \rho \rho^{n+1}=\Delta=xy$. For $x=\tau_2$ we have $\Delta=\tau_2 \rho^n \tau_1 \rho=xy$ with $y=\rho^n \tau_1 \rho$. We have $yx=\rho^n \tau_1 \rho \tau_2=\rho^n \rho^2=\Delta$. 
\end{proof}

We deduce:

\begin{prop}[New Garside structure for $G_{13}$ and dihedral Artin groups of even type]\label{dih_garside_thm}
The monoid defined by Presentation~\ref{dih_even} is a Garside monoid, with (central) Garside element $\Delta=\rho^{n+2}$. The corresponding Garside group is isomorphic to the dihedral Artin group of type $I_2(4+2n)$ (which is also isomorphic to the complex braid group of $G_{13}$ when $n=1$).
\end{prop}

\begin{proof} 
We argue exactly as in the proof of Theorem~\ref{main_text}. We have seen above that the monoid has Noetherian divisibility, and the various other preliminary results required to apply the same proof as in the aforementioned theorem were given in Lemma~\ref{dih_even}, Corollary~\ref{cor_dih_cancel} and Lemma~\ref{dih_garside_elt}.

The isomorphism with a dihedral Artin group was established in Lemma~\ref{isom_dih_even}. 
\end{proof}

\begin{question}
It would be interesting to understand for which (finite) complex reflection groups the Garside structure introduced in this paper can be constructed and whether it admits a canonical construction from the reflection group data. 
\end{question}

\begin{rmq}
Note that the dihedral Artin group of type $I_2(m)$, where $m\geq 3$ is odd, is isomorphic to $G(2,m)$. As for $G(n,n+1)$, the Garside structure obtained in this paper was already given for $G(2,m)$ in~\cite[Section 6]{Gobet}. 
\end{rmq}

\appendix

\section{Sharp cube condition for right-cancellativity}\label{sec_appendix}

This appendix is devoted to checking the sharp $\eta$-cube condition for the presentation $\langle \mathcal{T}~\vert~(\mathcal{R}'')^{\mathrm{op}}\rangle$ of $\mathcal{M}(n,m)^{\mathrm{op}}$ introduced in the paragraph after Corollary~\ref{cor_pres_right_c}. That is, we check that for every triple $(\tau_i, \tau_j, \tau_\ell)$ if pairwise distinct elements of $\mathcal{T}$, we have \begin{align}\label{eta_right}\eta(\eta(\tau_i, \tau_j), \eta(\tau_i, \tau_\ell))=\eta(\eta(\tau_j, \tau_i), \eta(\tau_j, \tau_\ell)).\end{align} 
Note that it suffices to threat the three cases $i<j<\ell$, $i<\ell<j$, and $\ell<j<i$, as the remaining cases are obtained for free by swapping the roles of $i$ and $j$. These three cases are established respectively in Lemmatas~\ref{right_cube_1}, \ref{right_cube_2} and~\ref{right_cube_3} below. Recall the values of the syntactic right-complement $\eta$ given in~\eqref{right_less} and~\eqref{right_great}.

\begin{lemma}\label{right_cube_1}
	The presentation $\langle \mathcal{T}~\vert~(\mathcal{R}'')^{\mathrm{op}}\rangle$ satisfies the sharp $\eta$-cube condition for every triple $(\tau_i, \tau_j, \tau_\ell)$ of pairwise distinct elements of $\mathcal{T}$ with $i<j<\ell$.
\end{lemma}
\begin{proof} 
	We will show that both sides of~\eqref{eta_right} are equal to $1$. We distinguish between various cases.
	\begin{itemize}
		\item Case $D(i,j)=D(j,\ell)=D(i,\ell)=1$.
		
		\noindent We have $\eta(\eta(\tau_i, \tau_j),\eta(\tau_i,\tau_\ell))=\eta((\tau_n^q\tau_1)^{n-j},(\tau_n^q\tau_1)^{n-\ell})=\eta((\tau_n^q\tau_1)^{\ell-j},1)=1$,
		
		\noindent while $\eta(\eta(\tau_j, \tau_i),\eta(\tau_j,\tau_\ell))=\eta((\tau_n^q\tau_1)^{n-j}\tau_{{n-j+i}},(\tau_n^q\tau_1)^{n-\ell})=\eta((\tau_n^q\tau_1)^{\ell-j}\tau_{n-j+i},1)=1$. 
		
		\item Case $D(i,j)=0$, $D(j,\ell)=D(i,\ell)=1$.
		
		\noindent We have $\eta(\eta(\tau_i, \tau_j),\eta(\tau_i,\tau_\ell))=\eta((\tau_n^q\tau_1)^{n-j+1},(\tau_n^q\tau_1)^{n-\ell})=\eta((\tau_n^q\tau_1)^{\ell-j+1},1)=1$,
		
		\noindent while $\eta(\eta(\tau_j, \tau_i),\eta(\tau_j,\tau_\ell))=\eta((\tau_n^q\tau_1)^{n-j}\tau_n^{q-1+D(n-j+i)} \tau_{{n-j+i+1}},(\tau_n^q\tau_1)^{n-\ell})=1$.
		
		\item Case $D(i,j)=1=D(i,\ell)$, $D(j,\ell)=0$. 
		
		\noindent We have $\eta(\eta(\tau_i, \tau_j),\eta(\tau_i,\tau_\ell))=\eta((\tau_n^q\tau_1)^{n-j},(\tau_n^q\tau_1)^{n-\ell})=\eta((\tau_n^q\tau_1)^{\ell-j},1)=1$,
		
		\noindent while $\eta(\eta(\tau_j, \tau_i),\eta(\tau_j,\tau_\ell))= \eta((\tau_n^q\tau_1)^{n-j}\tau_{{n-j+i}},(\tau_n^q\tau_1)^{n-\ell+1})=\eta((\tau_n^q\tau_1)^{\ell-j-1}\tau_{{n-j+i}},1)=1$.
		
		\item Cases $(D(i,j)=D(j,\ell)=1$, $D(i,\ell)=0)$ and $(D(i,j)=D(j,\ell)=0$, $D(i,\ell)=1)$. By Lemma~\ref{ft}~(2), these situations cannot appear. 
		
		\item Case $D(i,j)=D(i,\ell)=0$, $D(j,\ell)=1$.
		
		\noindent We have $\eta(\eta(\tau_i, \tau_j),\eta(\tau_i,\tau_\ell))=\eta((\tau_n^q\tau_1)^{n-j+1},(\tau_n^q\tau_1)^{n-\ell+1})=\eta((\tau_n^q\tau_1)^{\ell-j},1)=1$,
		
		\noindent while  $\eta(\eta(\tau_j, \tau_i),\eta(\tau_j,\tau_\ell))=\eta((\tau_n^q\tau_1)^{n-j}\tau_n^{q-1+D(n-j+i)} \tau_{{n-j+i+1}},(\tau_n^q\tau_1)^{n-\ell})=1.$
		
		\item Case $D(i,j)=1$, $D(j,\ell)=D(i,\ell)=0$.
		
		\noindent We have $\eta(\eta(\tau_i, \tau_j),\eta(\tau_i,\tau_\ell))=\eta((\tau_n^q\tau_1)^{n-j},(\tau_n^q\tau_1)^{n-\ell+1})=\eta((\tau_n^q\tau_1)^{\ell-j-1},1)=1$,
		
		\noindent while $\eta(\eta(\tau_j, \tau_i),\eta(\tau_j,\tau_\ell))=\eta((\tau_n^q\tau_1)^{n-j}\tau_{{n-j+i}},(\tau_n^q\tau_1)^{n-\ell+1})=\eta((\tau_n^q\tau_1)^{\ell-j-1}\tau_{n-j+i}, 1)=1$.
		
		\item Case $D(i,j)=D(j,\ell)=D(i,\ell)=0$.
		
		\noindent We have $\eta(\eta(\tau_i, \tau_j),\eta(\tau_i,\tau_\ell))=\eta((\tau_n^q\tau_1)^{n-j+1}, (\tau_n^q\tau_1)^{n-\ell+1})=\eta((\tau_n^q\tau_1)^{\ell-j},1)=1$,
		
		\noindent while  $\eta(\eta(\tau_j, \tau_i),\eta(\tau_j,\tau_\ell))=\eta((\tau_n^q\tau_1)^{n-j}\tau_n^{q-1+D(n-j+i)} \tau_{{n-j+i+1}},(\tau_n^q\tau_1)^{n-\ell+1})=1$.
\end{itemize} \end{proof}

\begin{lemma}\label{right_cube_2}
	The presentation $\langle \mathcal{T}~\vert~(\mathcal{R}'')^{\mathrm{op}}\rangle$ satisfies the sharp $\eta$-cube condition for every triple $(\tau_i, \tau_j, \tau_\ell)$ of pairwise distinct elements of $\mathcal{T}$ with $i<\ell<j$.
\end{lemma}

\begin{proof}
	We need to distinguish between various cases. Unlike in the proof of Lemma~\ref{right_cube_1}, the value of either side of~\eqref{eta_right} is not the same for all the cases.
	\begin{itemize}
		
		\item Case $D(i,\ell)=D(\ell,j)=D(i,j)=1$.
		
		\noindent We have $\eta(\eta(\tau_i, \tau_j),\eta(\tau_i,\tau_\ell))=\eta((\tau_n^q\tau_1)^{n-j},(\tau_n^q\tau_1)^{n-\ell})=\eta(1,(\tau_n^q\tau_1)^{j-\ell})=(\tau_n^q\tau_1)^{j-\ell},$
		
		\noindent while $\eta(\eta(\tau_j, \tau_i),\eta(\tau_j,\tau_\ell))=\eta((\tau_n^q\tau_1)^{n-j}\tau_{{n-j+i}},(\tau_n^q\tau_1)^{n-j}\tau_{n-j+\ell})=\eta(\tau_{n-j+i}, \tau_{n-j+\ell})$. By Lemma~\ref{ft}~(4) we have $D(n-j+i,n-j+\ell)=1$, hence $\eta(\tau_{n-j+i}, \tau_{n-j+\ell})=(\tau_n^q\tau_1)^{j-\ell}$.
		
		\item Case $D(i,\ell)=0$, $D(\ell,j)=D(i,j)=1$.
		
		\noindent We have $\eta(\eta(\tau_i, \tau_j),\eta(\tau_i,\tau_\ell))=\eta((\tau_n^q\tau_1)^{n-j},(\tau_n^q\tau_1)^{n-\ell+1})=\eta(1,(\tau_n^q\tau_1)^{j-\ell+1})=(\tau_n^q\tau_1)^{j-\ell+1},$
		
		\noindent while $\eta(\eta(\tau_j, \tau_i),\eta(\tau_j,\tau_\ell))=\eta((\tau_n^q\tau_1)^{n-j}\tau_{{n-j+i}},(\tau_n^q\tau_1)^{n-j}\tau_{{n-j+\ell}})=\eta(\tau_{n-j+i}, \tau_{n-j+\ell})$. By Lemma~\ref{ft}~(4) we have $D(n-j+i, n-j+\ell)=0$, hence $\eta(\tau_{n-j+i}, \tau_{n-j+\ell})=(\tau_n^q\tau_1)^{j-\ell+1}$. 
		
		\item Case $D(i,\ell)=1=D(i,j)$, $D(\ell,j)=0$.
		
		\noindent We have $\eta(\eta(\tau_i, \tau_j),\eta(\tau_i,\tau_\ell))=\eta((\tau_n^q\tau_1)^{n-j},(\tau_n^q\tau_1)^{n-\ell})=\eta(1,(\tau_n^q\tau_1)^{j-\ell})=(\tau_n^q\tau_1)^{j-\ell},$
		
		\noindent while \begin{align*}
			\eta(\eta(\tau_j, \tau_i),\eta(\tau_j,\tau_\ell))&= \eta((\tau_n^q\tau_1)^{n-j}\tau_{{n-j+i}},(\tau_n^q\tau_1)^{n-j} \tau_n^{q-1+D(n-j+\ell)}\tau_{{n-j+\ell+1}})\\
			&=\eta(\tau_{{n-j+i}},\tau_n^{q-1+D(n-j+\ell)}\tau_{{n-j+\ell+1}}).\end{align*}
		
		In the case $D(n-j+\ell)=1$, by Lemma~\ref{lem_tech_right}~(1) we have $\eta(\tau_{{n-j+i}},\tau_n^{q}\tau_{{n-j+\ell+1}})=(\tau_n^q\tau_1)^{j-\ell}$ except if $D(n-j+i)=0$ and $D(n-j+i+1,n-j+\ell+1)=0$. But by Lemma~\ref{ft}~(1) it cannot happen. In the case $D(n-j+\ell)=0$, by Lemma~\ref{lem_tech_right} we have $\eta(\tau_{{n-j+i}},\tau_n^{q-1}\tau_{{n-j+\ell+1}})=(\tau_n^q \tau_1)^{j-\ell}$ except possibly if $q=1$ and $D(n-j+i,n-j+\ell+1)=1$. But we claim that this last case cannot happen. Indeed, using $D(i,j)=1=D(i,\ell)$ and $D(\ell,j)=0$, a calculation yields $k_{i+n-\ell}=n-k_{\ell+n-j}+k_{i+n-j}.$ Subtracting $r$ in each side gives, using the equality $D(n-j+\ell)=0$, that $k_{i+n-\ell}-r=n-k_{\ell+n-j+1}+k_{i+n-j}>0$, hence $D(i+n-\ell-1)=0$. We then get that \begin{align*} n D(n-j+i, n-j+\ell+1)&=-k_{n-j+\ell+1}+n+k_{n-j+i}-k_{n+i-\ell-1}\\ &=-k_{n-j+\ell}+n+k_{n-j+i}-k_{n+i-\ell}=0, \end{align*} 
		yielding $D(n-j+i,n-j+\ell+1)=0$.

		\item Cases $(D(i,\ell)=D(\ell,j)=1, D(i,j)=0)$ and $(D(i,\ell)=0=D(\ell,j), D(i,j)=1)$. By Lemma~\ref{ft}~(2), these situations cannot appear.
		
		\item Case $D(i,\ell)=0=D(i,j)$, $D(\ell,j)=1$.
		
		\noindent We have $\eta(\eta(\tau_i, \tau_j),\eta(\tau_i,\tau_\ell))=\eta((\tau_n^q\tau_1)^{n-j+1}, (\tau_n^q\tau_1)^{n-\ell+1})=\eta(1,(\tau_n^q\tau_1)^{j-\ell})=(\tau_n^q\tau_1)^{j-\ell},$
		
		\noindent while \begin{align*}
			\eta(\eta(\tau_j, \tau_i),\eta(\tau_j,\tau_\ell))&= \eta((\tau_n^q\tau_1)^{n-j}\tau_n^{q-1+D(n-j+i)} \tau_{{n-j+i+1}},(\tau_n^q\tau_1)^{n-j}\tau_{{n-j+\ell}}) \\ &= \eta(\tau_n^{q-1+D(n-j+i)} \tau_{{n-j+i+1}},\tau_{{n-j+\ell}}).\end{align*} 
		Assume that $D(n-j+i)=0$. We then get $\eta(\tau_n^{q-1} \tau_{{n-j+i+1}},\tau_{{n-j+\ell}})=(\tau_n^q\tau_1)^{j-\ell}$ if $q\neq 1$ by Lemma~\ref{lem_tech_left}. If $q=1$, let us first consider the case $i+1\neq \ell$. Then we have $\eta(\tau_n^{q-1} \tau_{{n-j+i+1}},\tau_{{n-j+\ell}})=\eta(\tau_{{n-j+i+1}},\tau_{{n-j+\ell}})=(\tau_n^q\tau_1)^{j-\ell}$ except if $D(n-j+i+1,n-j+\ell)=0$. But in this case one gets, combining the conditions $D(n-j+i+1,n-j+\ell)=0$ and $D(n-j+i)=0$, that $n+k_{n-j+i}+r-k_{n-j+\ell}=k_{n+i+1-\ell},$ and combining the three conditions $D(i,\ell)=0=D(i,j)$, $D(\ell,j)=1$ yields $k_{n-j+i}-k_{n-j+\ell}=k_{n-i+\ell}$. We thus get $k_{n+i+1-\ell}=n+r+k_{i+n-\ell}>n,$ a contradiction. If $i+1=\ell$ then we have $D(i,i+1)=D(i,j)=0, D(i+1,j)=1$ which by Lemma~\ref{ft}~(5) yields $D(n-j+i)=1$, a contradiction. Now assume that $D(n-j+i)=1$. By Lemma~\ref{lem_tech_left}, we have $\eta(\tau_n^q \tau_{{n-j+i+1}}, \tau_{{n-j+\ell}})=(\tau_n^q\tau_1)^{j-\ell}$ except possibly if $D(n-j+\ell)=0$, $D(n-j+i+1,n-j+\ell+1)=1$. But since $D(n-j+i)=1$, by Lemma~\ref{ft}~(1) this cannot happen.

		\item Case $D(i,\ell)=1$, $D(\ell,j)=D(i,j)=0$. 
		
		\noindent We have $\eta(\eta(\tau_i, \tau_j),\eta(\tau_i,\tau_\ell))=\eta((\tau_n^q\tau_1)^{n-j+1},(\tau_n^q\tau_1)^{n-\ell})=\eta(1, (\tau_n^q\tau_1)^{j-\ell-1})=(\tau_n^q\tau_1)^{j-\ell-1},$
		
		\noindent while \begin{align*}
			\eta(\eta(\tau_j, \tau_i),\eta(\tau_j,\tau_\ell))&= \eta((\tau_n^q\tau_1)^{n-j}\tau_n^{q-1+D(n-j+i)} \tau_{{n-j+i+1}},(\tau_n^q\tau_1)^{n-j}\tau_n^{q-1+D(n-j+\ell)} \tau_{{n-j+\ell+1}})\\ &=\eta(\tau_n^{D(n-j+i)} \tau_{{n-j+i+1}},\tau_n^{D(n-j+\ell)} \tau_{{n-j+\ell+1}}). 
		\end{align*} 
		
		The conditions $D(i,\ell)=1$, $D(\ell,j)=0$, $D(i,j)=0$ imply that $k_{\ell+n-j}+k_{i+n-\ell}=k_{i+n-j}$, hence we have the implication $(D(\ell+n-j)=1\Rightarrow D(i+n-j)=1)$. It follows that the case $D(\ell+n-j)=1$, $D(i+n-j)=0$ cannot appear. In the case $D(\ell+n-j)=0$, $D(i+n-j)=1$, by Lemma~\ref{lem_tech_left} we get that $\eta(\tau_n \tau_{{n-j+i+1}}, \tau_{{n-j+\ell+1}})=(\tau_n^q\tau_1)^{j-\ell-1}$ except if $q=1$, $D(n-j+\ell+1)=0$ and $D(n-j+i+1, n-j+\ell+2)=1$. But by Lemma~\ref{ft}~(1) it cannot happen. Hence assume that $D(\ell+n-j)=D(i+n-j)$. Using this we obtain that $D(n-j+i+1,n-j+\ell+1)=D(n-j+i,n-j+\ell)$, which is equal to $D(i,\ell)=1$ by Lemma~\ref{ft}~(4). For $\varepsilon\in\{0,1\}$ we thus get $\eta(\tau_n^\varepsilon\tau_{{n-j+i+1}},\tau_n^\varepsilon\tau_{{n-j+\ell+1}})=\eta(\tau_{{n-j+i+1}},\tau_{{n-j+\ell+1}})=(\tau_n^q\tau_1)^{j-\ell-1}$.
		\item Case $D(i,\ell)=D(\ell,j)=D(i,j)=0$.
		
		\noindent We have $\eta(\eta(\tau_i, \tau_j),\eta(\tau_i,\tau_\ell))=\eta((\tau_n^q\tau_1)^{n-j+1},(\tau_n^q\tau_1)^{n-\ell+1})=\eta(1, (\tau_n^q\tau_1)^{j-\ell})=(\tau_n^q\tau_1)^{j-\ell},$
		
		\noindent while as in the previous case we have  \begin{align*}
			\eta(\eta(\tau_j, \tau_i),\eta(\tau_j,\tau_\ell))&= \eta(\tau_n^{D(n-j+i)} \tau_{{n-j+i+1}},\tau_n^{D(n-j+\ell)} \tau_{{n-j+\ell+1}}).\end{align*}
		The conditions $D(i,\ell)=D(\ell,j)=D(i,j)=0$ yield $k_{\ell+n-j}+k_{i+n-\ell}=k_{i+n-j}+n$, hence we have the implication $(D(\ell+n-j)=0\Rightarrow D(i+n-j)=0)$. It follows that the case $D(\ell+n-j)=0, D(i+n-j)=1$ cannot appear. In the case $D(\ell+n-j)=1, D(i+n-j)=0$, by Lemma~\ref{lem_tech_right} we get $\eta(\tau_{n-j+i+1}, \tau_n \tau_{n-j+\ell+1})=(\tau_n^q \tau_1)^{j-\ell}$ except in two cases, which in fact cannot appear: if $i+1=\ell$ and $D(n-j+i+1)=0$ (which yields $D(n-j+\ell)=0$, a contradiction), and if $i+1\neq \ell$, $D(n-j+i+1)=0$ and $D(n-j+i+2,n-j+\ell+1)=0$ (which contradicts Lemma~\ref{ft}~(1)).  
		Hence assume that $D(n-j+i)=D(n-j+\ell)$. We then easily obtain that $D(n-j+i+1,n-j+\ell+1)=D(n-j+i, n-j+\ell)$, which by Lemma~\ref{ft}~(4) is equal to $D(i,\ell)=0$. For $\varepsilon\in\{0,1\}$ we thus get $\eta(\tau_n^{\varepsilon}\tau_{{n-j+i+1}},\tau_n^{\varepsilon}\tau_{{n-j+\ell+1}})=\eta(\tau_{{n-j+i+1}},\tau_{{n-j+\ell+1}})=(\tau_n^q\tau_1)^{j-\ell}$.
	\end{itemize}
\end{proof}

\begin{lemma}\label{right_cube_3}
	The presentation $\langle \mathcal{T}~\vert~(\mathcal{R}'')^{\mathrm{op}}\rangle$ satisfies the sharp $\eta$-cube condition for every triple $(\tau_i, \tau_j, \tau_\ell)$ of pairwise distinct elements of $\mathcal{T}$ with $\ell<j<i$.
\end{lemma}

\begin{proof}
	We need to distinguish between various cases. As in the proof of Lemma~\ref{right_cube_2}, the value of either side of~\eqref{eta_right} is not the same for all the cases.
	\begin{itemize}
		\item Case $D(\ell,j)=D(j,i)=D(\ell,i)=1$. 
		
		\noindent We have $\eta(\eta(\tau_i, \tau_j),\eta(\tau_i,\tau_\ell))=\eta((\tau_n^q\tau_1)^{n-i}\tau_{n-i+j},(\tau_n^q\tau_1)^{n-i}\tau_{n-i+\ell})=\eta(\tau_{n-i+j}, \tau_{n-i+\ell}),$ and by Lemma~\ref{ft}~(4) we get $D(n-i+\ell,n-i+j)=1$, yielding $\eta(\tau_{n-i+j}, \tau_{n-i+\ell})=(\tau_n^q\tau_1)^{i-j} \tau_{n+\ell-j},$ while
		$$\eta(\eta(\tau_j, \tau_i),\eta(\tau_j,\tau_\ell))=\eta((\tau_n^q\tau_1)^{n-i},(\tau_n^q\tau_1)^{n-j}\tau_{n-j+\ell})=\eta(1, (\tau_n^q\tau_1)^{i-j}\tau_{n-j+\ell})=(\tau_n^q\tau_1)^{i-j}\tau_{n-j+\ell}.$$
		
		\item Case $D(\ell,j)=0$, $D(j,i)=D(\ell,i)=1$.
		
		\noindent We have $\eta(\eta(\tau_i, \tau_j),\eta(\tau_i,\tau_\ell))=\eta((\tau_n^q\tau_1)^{n-i}\tau_{n-i+j},(\tau_n^q\tau_1)^{n-i}\tau_{n-i+\ell})=\eta(\tau_{n-i+j}, \tau_{n-i+\ell}),$ and by Lemma~\ref{ft}~(4) we get $D(n-i+\ell,n-i+j)=0$, yielding \begin{align*}\eta(\tau_{n-i+j}, \tau_{n-i+\ell})=(\tau_n^q\tau_1)^{i-j}\tau_n^{q-1+D(n+\ell-j)} \tau_{n+\ell-j+1},\end{align*}
		\noindent while
		\begin{align*}\eta(\eta(\tau_j, \tau_i),\eta(\tau_j,\tau_\ell))&=\eta((\tau_n^q\tau_1)^{n-i}, (\tau_n^q\tau_1)^{n-j}\tau_n^{q-1+D(n+\ell-j)}\tau_{n-j+\ell+1})\\ &=(\tau_n^q\tau_1)^{i-j}\tau_n^{q-1+D(n+\ell-j)} \tau_{n+\ell-j+1}.
		\end{align*} 
		\item \noindent Case $D(\ell,j)=1=D(\ell,i)$, $D(j,i)=0$. 
		
		\noindent On one hand we have \begin{align*}\eta(\eta(\tau_i, \tau_j),\eta(\tau_i,\tau_\ell))&=\eta((\tau_n^q\tau_1)^{n-i} \tau_n^{q-1+D(n-i+j)} \tau_{n-i+j+1}, (\tau_n^q\tau_1)^{n-i}\tau_{n-i+\ell})\\
			&=\eta(\tau_n^{q-1+D(n-i+j)} \tau_{n-i+j+1}, \tau_{n-i+\ell}).
		\end{align*} 
		If $D(n-i+j)=0$, then by Lemma~\ref{tech_tn_3} we get for $q\neq 1$ that
		$$\eta(\tau_n^{q-1} \tau_{n-i+j+1}, \tau_{n-i+\ell})=(\tau_n^q\tau_1)^{i-j-1} \tau_n^{D(n-i+\ell+1,n-i+j+1)+D(n-i+\ell)}\tau_{n-j+\ell}=(\tau_n^q\tau_1)^{i-j-1}\tau_{n-j+\ell},$$ where for the last equality one proceeds as follows; using the conditions $D(\ell,j)=1=D(\ell,i)$, $D(j,i)=0$ together with $D(n-i+j)=0$ we get $D(n-i+\ell+1,n-i+j+1)=\frac{-r-k_{n-i+\ell}+k_{n-i+\ell+1}}{n}=-D(n-i+\ell)$. If $q=1$ then one similarly checks that $D(n-i+\ell, n-i+j+1)=-D(n+\ell-j-1)$ (which holds also for $q\neq 1$), which gives $D(n-i+\ell, n-i+j+1)=0=D(n+\ell-j-1)$ since both lie in $\{0,1\}$ (Lemma~\ref{lem_01}). We then conclude using~\eqref{right_great} that $\eta(\tau_{n-i+j+1}, \tau_{n-i+\ell})=(\tau_n^q\tau_1)^{i-j-1} \tau_n^{D(n+\ell-j-1)} \tau_{n+\ell-j}=(\tau_n^q\tau_1)^{i-j-1}  \tau_{n+\ell-j}.$  
		
		\noindent If $D(n-i+j)=1$, then one similarly checks that $D(n-i+\ell+1,n-i+j+1)=1-D(n-i+\ell)$. Hence by Lemma~\ref{tech_tn_3}, if $D(n-i+\ell)=1$ we get that $$\eta(\tau_n^{q} \tau_{n-i+j+1}, \tau_{n-i+\ell})=(\tau_n^q\tau_1)^{i-j-1} \tau_n^{D(n-i+\ell+1,n-i+j+1)}\tau_{n-j+\ell}=(\tau_n^q\tau_1)^{i-j-1}\tau_{n-j+\ell},$$ and if $D(n-i+\ell)=0$ we get also by Lemma~\ref{tech_tn_3} that $$\eta(\tau_n^{q} \tau_{n-i+j+1}, \tau_{n-i+\ell})=\eta(\tau_{n-i+j+1}, \tau_{n-i+\ell+1})=(\tau_n^q\tau_1)^{i-j-1} \tau_{n-j+\ell}.$$
		
		\noindent On the other hand we have $$\eta(\eta(\tau_j, \tau_i),\eta(\tau_j,\tau_\ell))=\eta(\tau_n^q \tau_1)^{n-i+1},(\tau_n^q\tau_1)^{n-j} \tau_{n-j+\ell})=\eta(1,(\tau_n^q\tau_1)^{i-j-1}\tau_{n-j+\ell})=(\tau_n^q\tau_1)^{i-j-1}\tau_{n-j+\ell}.$$

		\item Cases $(D(\ell,j)=D(j,i)=1, D(\ell,i)=0)$ and $(D(\ell,j)=0=D(j,i), D(\ell,i)=1)$. By Lemma~\ref{ft}~(2), these situations cannot appear.
			
		\item Case $D(\ell,j)=0=D(\ell,i)$, $D(j,i)=1$. We will show that  \begin{align*}\eta(\eta(\tau_i, \tau_j),\eta(\tau_i,\tau_\ell))=(\tau_n \tau_1)^{i-j} \tau_n^{q-1+D(n-j+\ell)}\tau_{n-j+\ell+1}=\eta(\eta(\tau_j, \tau_i),\eta(\tau_j,\tau_\ell))
		\end{align*} 
		On one hand we have \begin{align*}\eta(\eta(\tau_i, \tau_j),\eta(\tau_i,\tau_\ell))&=\eta((\tau_n^q\tau_1)^{n-i} \tau_{n-i+j}, (\tau_n^q\tau_1)^{n-i}\tau_n^{q-1+D(n-i+\ell)}\tau_{n-i+\ell+1}) \\
			&=\eta(\tau_{n-i+j}, \tau_n^{q-1+D(n-i+\ell)}\tau_{n-i+\ell+1}).\end{align*}
		Using the conditions $D(\ell,j)=0=D(\ell,i)$, $D(j,i)=1$ we obtain $k_{n-i+j}+k_{n-j+\ell}=k_{\ell+n-i}$. Hence if $D(n-i+\ell)=0$, we have $D(n-i+j)=0=D(n-j+\ell)$. In this case, if $q\neq 1$ then by Lemma~\ref{last_tech} we get \begin{align*} \eta(\tau_{n-i+j}, \tau_n^{q-1}\tau_{n-i+\ell+1})=(\tau_n \tau_1)^{i-j} \tau_n^{q-2+D(n-i+\ell+2, n-i+j+1)+D(n-i+\ell+1)} \tau_{n-j+\ell+1}.
		\end{align*} But using that $D(n-i+\ell)=D(n-i+j)=D(n-j+\ell)=0$ and $k_{n+\ell-i}=k_{n+j-i}+k_{n+\ell-j}$ we obtain $D(n-i+\ell+2, n-i+j+1)+D(n-i+\ell+1)=1$, yielding \begin{align*} \eta(\tau_{n-i+j}, \tau_n^{q-1}\tau_{n-i+\ell+1})=(\tau_n \tau_1)^{i-j} \tau_n^{q-1}\tau_{n-j+\ell+1}=(\tau_n \tau_1)^{i-j} \tau_n^{q-1+D(n-j+\ell)}\tau_{n-j+\ell+1}.
		\end{align*} If $q=1$ then we have $\eta(\tau_{n-i+j}, \tau_n^{q-1}\tau_{n-i+\ell+1})=\eta(\tau_{n-i+j}, \tau_{n-i+\ell+1}).$ We then similarly check that $D(n-i+\ell+1, n-i+j)=1$, in particular $j\neq \ell+1$, hence by~\eqref{right_great} we get \begin{align*}
			\eta(\tau_{n-i+j}, \tau_{n-i+\ell+1})=(\tau_n \tau_1)^{i-j} \tau_{n-j+\ell+1}=(\tau_n \tau_1)^{i-j} \tau_n^{q-1}\tau_{n-j+\ell+1}=(\tau_n \tau_1)^{i-j} \tau_n^{q-1+D(n-j+\ell)}\tau_{n-j+\ell+1}.
		\end{align*} If $D(n-i+\ell)=1$ then if $D(n-i+j)=0$ by Lemma~\ref{last_tech} we get $\eta(\tau_{n-i+j}, \tau_n^{q}\tau_{n-i+\ell+1})=\tau_n^q\tau_1 \eta(\tau_{n-i+j+1}, \tau_{n-i+\ell+1})$. Now Lemma~\ref{ft}~(1) yields $D(n-i+\ell+1, n-i+j+1)=0$, and by~\eqref{right_great} we get $\eta(\tau_{n-i+j+1}, \tau_{n-i+\ell+1})=(\tau_n^q\tau_1)^{i-j-1} \tau_n^{q-1+D(n-j+\ell)}\tau_{n-j+\ell+1}$.
	
		If $D(n-i+\ell)=1$ and $D(n-i+j)=1$ then by Lemma~\ref{last_tech} again we get \begin{align*} 
			\eta(\tau_{n-i+j}, \tau_n^q\tau_{n-i+\ell+1})=(\tau_n^q\tau_1)^{i-j}\tau_n^{q-1+D(n-i+\ell+2, n-i+j+1)+D(n-i+\ell+1)-D(n-i+j)}\tau_{n-j+\ell+1}.  
		\end{align*} Using $k_{n+\ell-i}=k_{n+j-i}+k_{n+\ell-j}$ and $D(n-i+\ell)=D(n-i+j)=1$ we obtain $D(n-i+\ell+2, n-i+j+1)+D(n-i+\ell+1)=1+D(n-j+\ell)$, yielding the expected value.
		
		\noindent On the other hand we have \begin{align*} 
			\eta(\eta(\tau_j, \tau_i), \eta(\tau_j, \tau_\ell))&=\eta((\tau_n^q\tau_1)^{n-i}, (\tau_n^q\tau_1)^{n-j} \tau_n^{q-1+D(n-j+\ell)}\tau_{n-j+\ell+1})\\ &=(\tau_n^q\tau_1)^{i-j} \tau_n^{q-1+D(n-j+\ell)}\tau_{n-j+\ell+1}.
		\end{align*} 
		\item Case $D(\ell,j)=1$, $D(\ell,i)=D(j,i)=0$. We show that $$\eta(\eta(\tau_i, \tau_j), \eta(\tau_i, \tau_\ell))=(\tau_n^q\tau_1)^{i-j-1}\tau_{n-j+\ell}=\eta(\eta(\tau_j, \tau_i), \eta(\tau_j, \tau_\ell)).$$
		
		\noindent On one hand we have \begin{align*}
			\eta(\eta(\tau_i, \tau_j), \eta(\tau_i, \tau_\ell))&=\eta((\tau_n^q\tau_1)^{n-i}\tau_n^{q-1+D(n-i+j)}\tau_{n-i+j+1},(\tau_n^q\tau_1)^{n-i}\tau_n^{q-1+D(n-i+\ell)}\tau_{n-i+\ell+1})\\ &=\eta(\tau_n^{D(n-i+j)}\tau_{n-i+j+1},\tau_n^{D(n-i+\ell)}\tau_{n-i+\ell+1}).
		\end{align*}
		Assume that $D(n-i+\ell)=0$. Combining the three conditions $D(\ell,j)=1$, $D(\ell,i)=D(j,i)=0$ we get that $k_{\ell+n-i}=k_{j+n-i}+k_{\ell+n-j}$. It follows that if $D(\ell+n-i)=0$, then $D(j+n-i)=0=D(\ell+n-j)$. Hence the case $D(n-i+\ell)=0$ and $D(n-i+j)=1$ cannot appear.

		Assume that $D(n-i+\ell)=D(n-i+j)$. We have \begin{align*} 
			n D(n-i+\ell+1, n-i+j+1)&= -k_{n-i+j+1}+n+k_{n-i+\ell+1}-k_{n+\ell-j}\\ &=-k_{n-i+j}+n+k_{n-i+\ell}-k_{n+\ell-j}=n,
		\end{align*} where the second equality is obtained using $D(n-i+\ell)=D(n-i+j)$. Hence we have $D(n-i+\ell+1, n-i+j+1)=1$ and by~\eqref{right_great} for $\varepsilon\in\{0,1\}$ we get $\eta(\eta(\tau_i, \tau_j), \eta(\tau_i, \tau_\ell))=\eta(\tau_n^{\varepsilon}\tau_{n-i+j+1},\tau_n^{\varepsilon}\tau_{n-i+\ell+1})=\eta(\tau_{n-i+j+1},\tau_{n-i+\ell+1})=(\tau_n^q\tau_1)^{i-j-1}\tau_{n-j+\ell}.$ 
		
		Now assume that $D(n-i+\ell)=1$ and $D(n-i+j)=0$. We have \begin{align*}
			&n(-D(n-i+j+1)+D(n-i+\ell+2, n-i+j+2)+D(n-i+\ell+1))\\ &=-(r+k_{n-i+j+1}-k_{n-i+j+2})-k_{n-i+j+2}+n+k_{n-i+\ell+2}-k_{n+\ell-j}+r+k_{n-i+\ell+1}-k_{n-i+\ell+2} \\&=-k_{n-i+j+1}+n+k_{n-i+\ell+1}-k_{n+\ell-j}=n D(n-i+\ell+1, n-i+j+1)=0,
		\end{align*}where the last equality is obtained by Lemma~\ref{ft}~(1). Hence if $q\neq 1$ or $q=1$ and $D(n-i+j+1)=1$, by Lemma~\ref{last_tech} we get $\eta(\tau_{n-i+j+1},\tau_n \tau_{n-i+\ell+1})=(\tau_n^q\tau_1)^{i-j-1}\tau_{n-j+\ell}$. It remains to threat the case when $q=1$ and $D(n-i+j+1)=0$. In this case by Lemma~\ref{last_tech} we have $\eta(\tau_{n-i+j+1}, \tau_n \tau_{n-i+\ell+1})=(\tau_n^q \tau_1)\eta(\tau_{n-i+j+2}, \tau_{n-i+\ell+1})$. Since $D(n-i+\ell)=1$ and $D(n-i+j+1)=0$, by Lemma~\ref{ft}~(1) we get $D(n-i+\ell+1, n-i+j+2)=0$, hence~\eqref{right_great} yields $\eta(\tau_{n-i+j+2}, \tau_{n-i+\ell+1})=(\tau_n^q\tau_1)^{i-j-2} \tau_n^{D(n+\ell-j-1)} \tau_{n-j+\ell}$. To conclude this case if therefore suffices to show that $D(n+\ell-j-1)=0$. This holds true as \begin{align*}
			0&=n D(n-i+\ell+1, n-i+j+2)=n + k_{n-i+\ell+1}-k_{n-i+j+2}-k_{n+\ell-j-1}\\ &=n-n+k_{n-i+\ell}+r - k_{n-i+j}-2r-k_{n+\ell-j-1}=k_{n+\ell-j}-r-k_{n+\ell-j-1}=-D(n+\ell-j-1). 
		\end{align*}  
		
		\noindent On the other hand we have \begin{align*}
			\eta(\eta(\tau_j, \tau_i), \eta(\tau_j, \tau_\ell))&=\eta((\tau_n^q\tau_1)^{n-i+1}, (\tau_n^q\tau_1)^{n-j} \tau_{n-j+\ell})=\eta(1, (\tau_n^q\tau_1)^{i-j-1}\tau_{n-j+\ell})\\&=(\tau_n^q\tau_1)^{i-j-1}\tau_{n-j+\ell}.
		\end{align*} 
		
		\item Case $D(\ell,j)=D(\ell,i)=D(j,i)=0$.
		
		\noindent On one hand, exactly the same computation as in the previous case gives $\eta(\eta(\tau_i, \tau_j), \eta(\tau_i, \tau_\ell))=\eta(\tau_n^{D(n-i+j)}\tau_{n-i+j+1},\tau_n^{D(n-i+\ell)}\tau_{n-i+\ell+1})$. Combining the conditions $D(\ell,j)=D(\ell,i)=D(j,i)=0$ yields $k_{\ell+n-i}+n=k_{j+n-i}+k_{\ell+n-j}$. It follows that if $D(j+n-i)=0$ or $D(\ell+n-j)=0$, then $D(\ell+n-i)=0$. In particular the case when $D(n-i+j)=0$ and $D(n-i+\ell)=1$ is excluded.
		
		Assume that $D(n-i+\ell)=D(n-i+j)$. We then have
		\begin{align*}	n D( n-i+\ell+1, n-i+j+1)&= -k_{n-i+j+1}+n+k_{n-i+\ell+1}-k_{n+\ell-j}\\ &=-k_{n-i+j}+n+k_{n-i+\ell}-k_{n+\ell-j}=0,
		\end{align*} where the second equality is obtained using the fact that $D(n-i+\ell)=D(n-i+j)$. Hence we have $D(n-i+j+1, n-i+\ell+1)=0$ and by~\eqref{right_great} for $\varepsilon\in\{0,1\}$ we get $\eta(\eta(\tau_i, \tau_j), \eta(\tau_i, \tau_\ell))=\eta(\tau_n^{\varepsilon}\tau_{n-i+j+1},\tau_n^{\varepsilon}\tau_{n-i+\ell+1})=\eta(\tau_{n-i+j+1},\tau_{n-i+\ell+1})=(\tau_n^q\tau_1)^{i-j-1}\tau_n^{q-1+D(n-j+\ell)}\tau_{n-j+\ell+1}.$
		
		Now assume that $D(n-i+\ell)=0$ and $D(n-i+j)=1$. We have \begin{align*}
			&n (D(n-i+\ell+2, n-i+j+1)+D(n-i+\ell+1))\\ &=n+k_{n-i+\ell+2}-k_{n-i+j+1}-k_{n-j+\ell+1}+r+k_{n-i+\ell+1}-k_{n-i+\ell+2}\\ &=n-k_{n-i+j}-r+n-k_{n-j+\ell+1}+r+r+k_{n-i+\ell}\\ &=n+r+k_{n-j+\ell}-k_{n-j+\ell+1}=n (1+D(n+\ell-j)),
		\end{align*} hence if $q\neq 1$ or ($q=1$ and $D(n-i+\ell+1)=1$) by Lemma~\ref{tech_tn_3}~(2) we get \begin{align*}\eta(\tau_n \tau_{n-i+j+1}, \tau_{n-i+\ell+1})=(\tau_n^q\tau_1)^{i-j-1} \tau_n^{q-1+D(n-j+\ell)} \tau_{n-j+\ell+1}.
		\end{align*} It remains to threat the case when $q=1$ and $D(n-i+\ell+1)=0$. In this case by Lemma~\ref{tech_tn_3} we have $\eta(\tau_n \tau_{n-i+j+1}, \tau_{n-i+\ell+1})=\eta(\tau_{n-i+j+1}, \tau_{n-i+\ell+2})$ and $j\neq \ell+1$ as $D(n-i+\ell+1)\neq D(n-i+j)$. We have \begin{align*}
			&n D(n-i+\ell+2, n-i+j+1)=k_{n-i+\ell+2}+n-k_{n-i+j+1}-k_{n-j+\ell+1}\\ &=k_{n-i+\ell}+r-k_{n-i+j}-k_{n-j+\ell+1}=r+k_{n+\ell-j}-n -k_{n-j+\ell+1}= n(D(n+\ell-j)+1),
		\end{align*}where the second equality is obtained using $D(n-i+\ell)=0=D(n-i+\ell+1)$ and $D(n-i+j)=1$. This forces $D(n-i+\ell+2,n-i+j+1)=1$ and $D(n+\ell-j)=0$ since they line in~$\{0,1\}$ (Lemma~\ref{lem_01}). By~\eqref{right_great} we get \begin{align*}
			\eta(\tau_{n-i+j+1}, \tau_{n-i+\ell+2})=(\tau_n^q \tau_1)^{i-j-1} \tau_{n-j+\ell+1}=(\tau_n^q \tau_1)^{i-j-1} \tau_n^{q-1+D(n+\ell-j)}\tau_{n-j+\ell+1}.
		\end{align*}

			\noindent On the other hand we have \begin{align*} 
				\eta(\eta(\tau_j, \tau_i),\eta(\tau_j, \tau_\ell))&=\eta((\tau_n^q\tau_1)^{n-i+1}, (\tau_n^q\tau_1)^{n-j}\tau_n^{q-1+D(n-j+\ell)}\tau_{n-j+\ell+1})\\ &=\eta(1,(\tau_n^q\tau_1)^{i-j-1}\tau_n^{q-1+D(n-j+\ell)}\tau_{n-j+\ell+1})\\ &=(\tau_n^q\tau_1)^{i-j-1}\tau_n^{q-1+D(n-j+\ell)}\tau_{n-j+\ell+1}.	
			\end{align*}
	\end{itemize}\end{proof}

\end{document}